\documentclass[11pt]{amsart}

\usepackage{a4wide}
\usepackage{graphicx}

\usepackage[english]{babel}
\usepackage{amsmath,amsthm}
\usepackage{amsfonts}
\usepackage{amssymb}
\usepackage{enumitem}
\usepackage{mathtools}
\usepackage{graphicx}
\usepackage{wrapfig}
\usepackage{tikz-cd}
\usepackage{tikz}
\usepackage{xcolor,colortbl}
\usepackage{tabularx}
\usepackage{float}
\usetikzlibrary{arrows.meta,arrows}
\usepackage[colorlinks=true,linkcolor=blue]{hyperref}

\usetikzlibrary{decorations.markings}
\tikzset{double line with arrow/.style args={#1,#2}{decorate,decoration={markings,%
			mark=at position 0 with {\coordinate (ta-base-1) at (0,1pt);
				\coordinate (ta-base-2) at (0,-1pt);},
			mark=at position 1 with {\draw[#1] (ta-base-1) -- (0,1pt);
				\draw[#2] (ta-base-2) -- (0,-1pt);
}}}}
\tikzset{Equal/.style={-,double line with arrow={-,-}}}

\newcommand{\hooklongrightarrow}{\lhook\joinrel\longrightarrow}

\newcommand{\vphan}{\vphantom{\lambda\mu\eta\kappa}}

\newcommand{\N}{\mathbb{N}}
\newcommand{\Z}{\mathbb{Z}}

\DeclareMathOperator{\dd}{d}

\DeclareMathOperator{\Hom}{Hom}
\DeclareMathOperator{\res}{res}

\DeclareMathOperator{\Cohom}{H}
\DeclareMathOperator{\HH}{H}

\DeclareMathOperator{\GL}{GL}
\newcommand{\norm}{\mathcal{N}}

\DeclarePairedDelimiter{\abs}{\lvert}{\rvert}
\DeclarePairedDelimiter{\gen}{\langle}{\rangle}
\newcommand{\isom}{\cong}

\newcommand{\overbar}[1]{\mkern 1.5mu\overline{\mkern-1.5mu#1\mkern-1.5mu}\mkern 1.5mu}

\theoremstyle{plain}
\newtheorem{thm}{Theorem}
\newtheorem{cor}[thm]{Corollary}
\newtheorem{Conjecture}{Conjecture}
\newtheorem{lemma}{Lemma}[section]
\newtheorem{proposition}[lemma]{Proposition}
\newtheorem{theorem}[lemma]{Theorem}
\newtheorem{corollary}[lemma]{Corollary}

\theoremstyle{definition}

\newtheorem{remark}[lemma]{Remark}
\newtheorem{remarks}[lemma]{Remarks}

\newtheorem*{remark*}{Comment}

\DeclareMathOperator{\Heis}{Heis}

\title[Computing a spectral sequence of finite Heisenberg groups]{Computing a spectral sequence of finite Heisenberg groups of prime power order}

\author{Oihana Garaialde Oca\~na and Lander Guerrero S\'anchez}

\begin{document}

\subjclass[2000]{18G40, 20D15,20J06, 13A02}
\keywords{Spectral sequence, cohomology ring, finite $p$-groups}

\thanks{The authors were partially supported by the Basque Government's project IT483-22 and the  Spanish Government's project PID2020-117281GB-I00. The first author was partially supported by the Spanish Government's project MINECOG19/P39 and  the second author was partially supported by the University of the Basque Country's predoctoral fellowship  PIF19/44}

\maketitle

\begin{abstract}
	Let $p\geq 5$ be a prime number, let $n\geq 2$ be a natural number and let $\Heis(p^n)$ denote the Heisenberg group modulo $p^n$. We study the Lyndon-Hochschild-Serre spectral sequence $E(\Heis(p^n))$ associated  to $\Heis(p^n)$ considered as a split extension, and show that, $E(\Heis(p^n))$ collapses in the third page. Moreover, for a fixed $p$, the spectral sequences $E(\Heis(p^n))$ are isomorphic from the second page on.
\end{abstract}


\section{Introduction}





Group cohomology provides a framework to analyse intrinsic algebraic properties of a given group (see \cite[Section 2.1]{Evens91}, \cite{Quillen71} for instance) or to study automorphisms  of groups (compare \cite{Gaschutz65}, \cite{Gaschutz66} and \cite{Wells71}) and it also has applications in algebra and number theory (see \cite{Guillot18} and references therein).
It is also interesting to know which type of graded rings can occur as cohomology rings of finite groups and how many of them are distinct (compare \cite{Carlson05}, \cite{DGG17}, \cite{Symonds21}). However, computing cohomology is extremely complicated  and thus, there are few examples of such rings in the literature. One of the most powerful tools in computing such rings is the Lyndon-Hochschild-Serre spectral sequence (LHSss, for short) and in this paper, we  provide one of the first infinite families of groups of prime power order, whose associated LHSss collapse in the same page. More precisely,  let $p$ denote an odd prime number, let $n\geq 1$ be an integer and let $$G=\Heis(p^n)=C_{p^n}\ltimes (C_{p^n}\times C_{p^n})$$ be the Heisenberg group modulo $p^n$. Note that $G$ is just a finite quotient of the infinite Heisenberg group $\widehat{G}=\Z\ltimes (\Z\times \Z)$. Let moreover $K$ denote a field of characteristic $p$ with trivial $G$-action and let $\HH^{\bullet}(G)=\HH^{\bullet}(G;K)$ denote the cohomology ring of $G$ with coefficients in $K$. We study the LHSss $E$ associated to $G$ as a split extension of $C_{p^n}$ by $C_{p^n}\times C_{p^n}$. We show that,  for all prime numbers $p\geq 5$ and integers $n\geq2$, the spectral sequence $E$ collapses in the third page, and that for such fixed $p$, the spectral sequences $E$ are isomorphic from the second page on; independently of $n$.
To obtain that result, we follow  Siegel's techniques \cite{Siegel96}, where he computes the spectral sequence associated to $\Heis(p)$.

We summarise the main results and give an outline of the paper below. We start by setting the notation in Section \ref{sec: notation}, and in Section \ref{sec: SecondPage} we describe the additive and multiplicative structure of the second page $E_2$ of the spectral sequence $E$ (see Proposition \ref{Prop: Basis D2} and  Theorem \ref{Thm: Structure E2}).  In Section \ref{sec: survivetoE3}, we use maps between cohomology rings to detect some of the generators in $E_2$ that survive to the infinity page $E_{\infty}$.  In Section \ref{sec: SiegelGeneralization}, we provide a generalization of \cite[Corollary 2]{Siegel96}; being the key step to deduce the image of the second differential of the remaining generators in $E_2$ (see Theorem \ref{Thm: Charlap Vasquez cyclic} and Propositions \ref{Prop: d2 mu} and \ref{prop: d2 nu3}, respectively). We postpone the statement of Theorem \ref{Thm: Charlap Vasquez cyclic} to Section \ref{sec: SiegelGeneralization}, as it requires introducing a considerable amount of notation, and its proof can be found in Appendix \ref{Appendix: Siegel}. In Sections \ref{sec: DescriptionThirdPage} and \ref{sec: InfinityPage}, we describe the third page $E_3$ of the spectral sequence and we show that all the remaining differentials are trivial. In turn, we attain the main result of this paper. 

\begin{thm}\label{thm: TheoremA} Let $p\geq 5$ be a prime number and let $n\geq 2$ be an integer. Then, the following statements hold: 
\begin{enumerate}
\item[(i)] The LHS spectral sequence $E(\Heis(p^n))$ collapses in the third page. 
\item[(ii)] For a fixed prime number $p$, the spectral sequences $E(\Heis(p^n))$ are isomorphic as bigraded $K$-algebras from the second page on.
\end{enumerate}
\end{thm}
The description of the infinity page $E_{\infty}(\Heis(p^n))$ determines the dimension of the $K$-vector space $\HH^k\big(\Heis(p^n)\big)$, for every $k\geq 0$. Therefore, we obtain the principal result of Section \ref{sec: PoincareSeries}.

\begin{cor} Let $p\geq 5$ be a prime number and let $n\geq 2$ be an integer. Then, the Poincar\'e series of $\Cohom^{\bullet}(\Heis(p^n))$ is given as follows:
	\begin{displaymath}
	P(t)=\frac{1+t^2-t^3+t^4-t^5+t^{2p+1}}{(1-t)^2(1-t^{2p})}.
	\end{displaymath}
\end{cor}

\vspace{2.0mm}
In Section \ref{sec: conclusion}, we consider the case where $K$ is a finite field of characteristic $p$ and we obtain the next result (see Corollary \ref{cor: conclusion}). 

\begin{cor} Let $p\geq 5$ be a prime number and assume that $K$ is a finite field of characteristic $p$. Then,  there are only finitely many isomorphism types of (graded) algebras in the infinite collection $\{\HH^{\bullet}(\Heis(p^n))\}_{n\geq 1}$.
\end{cor}

The above result is not surprising as the rank of $G$ is two (see \cite{Symonds21}), and it also motivates us to state a conjecture (see Conjecture \ref{conjecture}).


\vspace{2.0mm}
{\it Acknowledgements.} We would like to thank S. F. Siegel for clarifying how to compute the equalities in 
Proposition \ref{prop: alpha and tau appendix}. We would also like to thank Jon Gonz\'alez-S\'anchez for the interesting conversations regarding this project and for his support. 

\section{Background and notation}\label{sec: notation}

Throughout, let $p$ denote an odd prime number, let $n\geq 1$ be an integer and let $K$ denote a field of characteristic $p$. We write $G=C_{p^n}\ltimes (C_{p^n}\times C_{p^n})$ for the Heisenberg group modulo $p^n$ and we set $M=C_{p^n}\times C_{p^n}=\gen{a,b}$ and $Q=C_{p^n}=\gen{\sigma}$. Note that the element $\sigma\in Q$ acts (on the right) on $M$ via  $a^{\sigma}=ab$ and $b^{\sigma}=b$. 

The cohomology ring of $M$ with coefficients in $K$ is
\begin{displaymath}
\Cohom^{\bullet}(M)=\Lambda(x_1,y_1)\otimes_K K[x_2,y_2]=\Lambda(x_1,y_1)\otimes K[x_2,y_2],
\end{displaymath}
with $\abs{x_i}=\abs{y_i}=i$, for $i=1,2$ (see \cite[Proposition 4.5.4]{CarlsonBook03}). We can take 
\begin{alignat*}{2}
x_1&= a^*, \quad\, & y_1 & =b^*, \\
x_2&=\beta_n( x_1), \quad\, & y_2 & =\beta_n( y_1),
\end{alignat*}
where $(\cdot)^*$ denotes the dual element and $\beta_n\colon \Cohom^1(M) \longrightarrow \Cohom^2(M)$ is the $n$-th Bockstein homomorphism \cite[Section 6.2, p.197]{McCleary01}.
The (left) action of $\sigma$ on 
$\Cohom^{\bullet}(M)$ can be shown to be given by
\begin{alignat}{2}\label{ActionOnCohomologyOfM}
\sigma \cdot x_1 & = x_1, \quad\, & \sigma \cdot y_1 & = x_1 + y_1, \\
\sigma \cdot x_2 & = x_2, \quad\, & \sigma \cdot y_2 & = x_2 + y_2. \nonumber
\end{alignat}

For a group $\tilde{G}$ with normal subgroup $\tilde{M}$, there exists a first quadrant spectral sequence $E(\tilde G)$ converging to $\HH^{\bullet}(\tilde{G})$ (see \cite[Section 7.2]{Evens91} and references therein). It is called the Lyndon-Hochschild-Serre spectral sequence (LHSss, for short), and satisfies that 
\[
E_2^{r,s}(\tilde{G})=\HH^r\big(\widetilde{M}; \HH^s(\widetilde{Q})\big) \Longrightarrow \HH^{r+s}(\tilde{G}),
\]
with $r,s,\geq 0$. In the case under study, $M$ is a normal subgroup of $G$ with quotient $Q$ and, for simplicity, we will denote by $E$  the LHSss  associated to the split extension 
\begin{equation}\label{eq: splitextension}
1\to  M\to  G \to  Q \to 1.
\end{equation}

\section{Description of the second page of the spectral sequence}\label{sec: SecondPage}

	

We follow the notation in the previous section and unless otherwise stated, we additionally assume until the end of the manuscript that $n\geq 2$. We use the minimal $KQ$-resolution (\cite[Section I.6]{Brown82}) to compute the cohomology groups $E_2^{r,s}$. Let $N(\sigma)=\sum_{i=0}^{p^n-1}\sigma^i\in KM$ and  as $n\geq 2$, it can be readily checked that, for all $\varphi\in \Cohom^{\bullet}(M)$,  $\sigma^{p}\cdot \varphi=\varphi$ and $N(\sigma)\cdot \varphi=0$ hold. The second page of the spectral sequence then takes the following form: 

\begin{displaymath}
E_2^{r,s}=\Cohom^r\big(Q;\Cohom^s(M)\big) \isom 
\begin{dcases}
\Cohom^s(M)^{Q}, & \text{if } r \text{ is even}, \\
\frac{\Cohom^s(M)}{(\sigma-1)\cdot \Cohom^s(M)}, & \text{if } r \text{ is odd}.
\end{dcases}
\end{displaymath}

Let now
\begin{displaymath}
z_{2p}=\prod_{i=0}^{p-1}\sigma^i\cdot y_2=\prod_{i=0}^{p-1}(ix_2+y_2)\in\Cohom^{2p}(M),
\end{displaymath}
and observe that the element $z_{2p}$ is invariant under the action of $\sigma$. Furthermore, if we 
write 
\begin{align*}
W &=\Lambda^{\bullet}[x_1,y_1]\otimes \gen{x_2^i y_2^j\mid i\geq 0, \quad 0\leq j<p}, 
\\
D_2^{r,\bullet}&=\begin{dcases}
W^{Q}, & \text{if } r \text{ is even}, \\
\frac{W}{(\sigma-1)\cdot W}, & \text{if } r \text{ is odd},
\end{dcases}
\end{align*}
we have that $\Cohom^{\bullet}(M)=K[z_{2p}]\otimes W$, and so
\begin{equation}\label{eq: E2tensorproduct}
E_2^{r,\bullet}=K[z_{2p}]\otimes D_2^{r,\bullet}=\begin{dcases}
K[z_{2p}]\otimes W^{Q}, & \text{if } r \text{ is even}, \\
K[z_{2p}]\otimes \frac{W}{(\sigma-1)\cdot W}, & \text{if } r \text{ is odd}.
\end{dcases}
\end{equation}
Consequently, it suffices to study the structure of $D_2$ so that the structure of $E_2$ is determined.

\subsection{Additive structure}

The first step will be determining a basis of the $K$-vector space $D_2^{r,s}$ for each $r,s\geq 0$.

\begin{proposition}\label{Prop: Basis D2}~
	\begin{enumerate} 
		\item[(i)] For $s\geq 0$, the basis elements of $(W^s)^{Q}$ are the following:
		\renewcommand{\arraystretch}{1.5}
\begin{small}		\begin{figure}[H]
			\centering
		\begin{tabular}{r|l|}
			 \hline
			$2i+1\geq 1$ & $x_2^i, \quad x_1y_1x_2^{i-1}$ \\ \hline
			$2i\geq 2$ & $x_1x_2^i, \quad (x_1y_2-y_1x_2)x_2^{i-1}$ \\ \hline
			$0$ & $1$ \\ \hline
			$s$ & $(W^s)^{Q}$
		\end{tabular}
		\end{figure}\end{small}
		\renewcommand{\arraystretch}{1}

		\item[(ii)] For $s\geq 1$, the basis elements of $(\sigma-1)W^s$ are the following:
		\renewcommand{\arraystretch}{1.5}
	\begin{footnotesize}	\begin{figure}[H]
			\centering
			\begin{tabular}{r|l|} \hline
				$2i+1\geq 3$ & $x_1x_2^jy_2^k, \text{ with } j\geq 1,\quad  0\leq k\leq p-2,\quad  j+k=i$ \\ 
				& $y_1x_2^jy_2^k, \text{ with } j\geq 2,\quad  0\leq k\leq p-3,\quad  j+k=i$\\
				& $x_1x_2^jy_2^k+ky_1x_2^{j+1}y_2^{k-1}, \text{ with } j\geq 0,\quad  0\leq k\leq p-1,\quad  j+k=i$ \\ \hline
				$2i\geq 2$ & $x_2^jy_2^k, \text{ with } j\geq 1,\quad  0\leq k\leq p-2,\quad  j+k=i$ \\ 
				& $x_1y_1x_2^jy_2^k, \text{ with } j\geq 1,\quad  0\leq k\leq p-2,\quad  j+k+1=i$\\ \hline
				$1$ & $x_1$ \\ \hline
				$s$ & $(\sigma-1)W^s$
			\end{tabular}
		\end{figure}\end{footnotesize}
		\renewcommand{\arraystretch}{1}

		\item[(iii)] For $s\geq 0$, the basis elements of $W^s/(\sigma-1)W^s$ are the following:
		\renewcommand{\arraystretch}{1.5}
	\begin{footnotesize}	\begin{figure}[H]
			\centering
			\begin{tabular}{r|l|} \hline
				$2i+1\geq 3$ & $\overbar{x_1^{\varepsilon}y_1^{1-\varepsilon}y_2^k}, \text{ with } \varepsilon=0,1,\quad  0\leq k\leq p-1,\quad  k=i$ \\ 
				& $\overbar{x_1^{\varepsilon}y_1^{1-\varepsilon}x_2^jy_2^{p-1}}, \text{ with } \varepsilon=0,1,\quad  j\geq 1,\quad  j+k+p-1=i$ \\ \hline
				$2i\geq 2$ & $\overbar{(x_1y_1)^{\varepsilon}y_2^k}, \text{ with } \varepsilon=0,1,\quad  0\leq k\leq p-1,\quad  \varepsilon+k=i$ \\ 
				& $\overbar{(x_1y_1)^{\varepsilon}x_2^jy_2^{p-1}}, \text{ with } \varepsilon=0,1,\quad  j\geq 0,\quad  \varepsilon + j+k+p-1=i$\\ \hline
				$1$ & $\overbar x_1$ \\ \hline
				$0$ & $\bar 1$ \\ \hline
				$s$ & $W^s/(\sigma-1)W^s$
			\end{tabular}
		\end{figure}\end{footnotesize}
		\renewcommand{\arraystretch}{1}
	\end{enumerate}
\end{proposition}
\begin{proof}
	The proof follows verbatim that of \cite[Proposition 3]{Siegel96}.
\end{proof}

Using this result, we can write a table with the basis elements of $D_2^{r,s}$:
\renewcommand{\arraystretch}{1.6}
\begin{small}\begin{figure}[H]
	\centering
	\begin{tabular}{r|c|c|} \hline
		$2i+1 \geq 2p+1$ & $x_1x_2^i, \ (x_1y_2-y_1x_2)x_2^{i-1}$ & $\overbar{x_1 x_2^{k-p+1} y_2^{p-1}}, \ \overbar{y_1 x_2^{k-p+1} y_2^{p-1}}$ \\ \hline
		$2i \geq 2p$ & $x_2^i, \ x_1y_1x_2^{i-1}$ & $\overbar{x_1y_1 x_2^{k-p}y_2^{p-1}}, \ \overbar{x_2^{k-p+1}y_2^{p-1}}$ \\ \hline
		$2i+1\leq 2p-1$	  & $x_1x_2^i, \ (x_1y_2-y_1x_2)x_2^{i-1}$ & $\overbar{x_1y_2^i},\  \overbar{y_1y_2^i}$ \\ \hline
		$2i\leq 2p-2$	  & $x_2^i, \ x_1y_1x_2^{i-1}$ & $\overbar{x_1y_1y_2^{i-1}},\  \overbar{y_2^i}$ \\ \hline
		$1$	  	  & $x_1$ & $\overbar{y_1}$ \\ \hline
		$0$	  	  & $1$ & $\bar 1$ \\ \hline 
		$s$ & $D_2^{2j,s}=(W^s)^{Q}$ & $D_2^{2j+1,s}=W^s/(\sigma-1)W^s$
	\end{tabular}
	\caption{Basis of $D_2^{r,s}$ for $r,s\geq 0$, with $j\geq 0$.}
\end{figure}\end{small}
\renewcommand{\arraystretch}{1}

\subsection{Multiplicative structure}

Following \cite[Section 4]{Siegel96} (see also \cite[Sections 3.2 and 7.3]{Evens91}) and using the diagonal approximation, we describe the multiplicative structure of $E_2$, that is, the bigraded algebra structure of $E_2$ over $K$. For $r,s,r',s'\geq 0$, let  $\varphi\in \Cohom^s(M)$ and $\varphi'\in \Cohom^{s'}(M)$ represent the elements $\bar \varphi\in E_2^{r,s}$ and $\bar\varphi'\in E_2^{r',s'}$, respectively. 
Then, their product in $E_2$ is the element $\bar\varphi \bar\varphi'\in E_2^{r+r',s+s'}$ with 
\begin{displaymath}
(-1)^{r's}\bar\varphi \bar\varphi' = 
\begin{dcases}
\overbar{\varphi\smallsmile\varphi'}, & \text{if } r \text{ or } r' \text{ is even}, \\
\sum_{0\leq i<j<p^n} \overbar{\sigma^i\cdot\varphi \smallsmile \sigma^j\cdot \varphi'}, & \text{if } r \text{ and } r' \text{ are odd}.
\end{dcases}
\end{displaymath}


\begin{lemma}\label{Lemm: Product E2 0}
                 Let $\bar \varphi\in E_2^{r,s}$ and $\bar\varphi'\in E_2^{r',s'}$ be as above with $r$ and $r'$ odd. Then, $\bar\varphi \bar\varphi' =0$.
\end{lemma}
\begin{proof}
	For simplicity, write
	$$N_0(\sigma)=0,\quad \text{and for } k\geq 1, \; N_k(\sigma)=\sum_{i=0}^{k-1}\sigma^i.$$
	In particular, we have that $N(\sigma)=N_{p^n}(\sigma)$. 
	Furthermore, note that, for $0\leq i\leq p-1$ and $k\geq 1$, we have that
	$$\sigma^{i+kp}\cdot \varphi=\sigma^i\cdot \varphi \; \text{ and }\; N_{i+kp}(\sigma)\cdot \varphi =N_{i}(\sigma)\cdot \varphi +kN_{p}(\sigma)\cdot \varphi.$$ 
	Then, we compute 
\begin{small}	\begin{align*}
	\sum_{0\leq i<j<p^n} \sigma^i\cdot \varphi \smallsmile \sigma^j\cdot \varphi' &= \sum_{j=0}^{p^n-1} N_{j}(\sigma)\cdot \varphi \smallsmile \sigma^j\cdot \varphi'= \sum_{j=0}^{p-1}\Bigg(\sum_{k=0}^{p^{n-1}-1}N_{j+kp}(\sigma)\Bigg) \cdot \varphi \smallsmile \sigma^j\cdot \varphi' \\
	& = \sum_{j=0}^{p-1}\Bigg(\sum_{k=0}^{p^{n-1}-1}N_{j}(\sigma)+kN_{p}(\sigma)\Bigg)\cdot  \varphi \smallsmile \sigma^j\cdot \varphi'\\
	& = \sum_{j=0}^{p-1}\bigg(p^{n-1}N_{j}(\sigma)+p^{n-1}\frac{(p^{n-1}-1)}{2}N_{p}(\sigma)\bigg)\cdot \varphi \smallsmile \sigma^j\cdot \varphi' =0 
	\end{align*}\end{small}
	As a consequence, $\bar\varphi \bar\varphi' =0$ in $E_2$.
\end{proof}

In order to describe the multiplicative structure of $E_2$, we fix the following notation.
\begin{gather*}
\lambda_1=x_1\in E_2^{0,1}, \qquad \lambda_2=x_2\in E_2^{0,2},\\
\nu_2=x_1y_1\in E_2^{0,2}, \qquad \nu_3=x_1y_2-y_1x_2\in E_2^{0,3}, \qquad \nu_{2p}=z_{2p}\in E_2^{0,2p},\\
\gamma_1=\bar{1}\in E_2^{1,0}, \qquad \gamma_2=\bar{1}\in E_2^{2,0}, \\
\text{for }1\leq i\leq p, \qquad \mu_{2i}=\overbar{y_1y_2^{i-1}}\in E_2^{1,2i-1}, \\
\text{for }1\leq i \leq  p-1, \qquad \mu_{2i+1}=\overbar{y_2^{i}}\in E_2^{1,2i}.
\end{gather*}

\begin{proposition} \label{Prop: Multiplication generators E2}
	Multiplication by the elements $\nu_{2p},\gamma_2, \lambda_2$ induces vector space homomorphisms as follows:
	\begin{enumerate}
		\item[(i)] Multiplication $\cdot\nu_{2p}\colon E_2^{r,s}\longrightarrow E_2^{r,s+2p}$ is injective for all $r,s\geq 0$.
		
		\item[(ii)] Multiplication $\cdot\gamma_2\colon E_2^{r,s}\longrightarrow E_2^{r+2,s}$ is  an isomorphism for all $r,s\geq 0$.

		\item[(iii)] Multiplication $\cdot\lambda_2\colon D_2^{r,s}\longrightarrow D_2^{r,s+2}$ is 
		an isomorphism 
		for all $s\geq 2p-1$.
	\end{enumerate}
\end{proposition}
\begin{proof} The first claim follows from Equation \eqref{eq: E2tensorproduct}. Using the identifications in Proposition \ref{Prop: Basis D2}, note that multiplication by $\gamma_2=\bar{1}$ is simply the identity homomorphism and so, the second item holds. The last statement is clear by the description of the bases in Proposition \ref{Prop: Basis D2}
%
%
\end{proof}

Using the previous results, we can deduce the multiplicative structure of $E_2$.

\begin{theorem}\label{Thm: Structure E2}
	The structure of the second page can be described as follows:
	\begin{enumerate}
		\item [(i)]The graded commutative algebra structure of the zeroth column is given by the following tensor product:
		\begin{displaymath}
		E_2^{0,\bullet}=K[\nu_{2p}]\otimes K[\lambda_1,\lambda_2,\nu_2,\nu_3]/(\nu_2^2,\lambda_1\nu_2,\nu_2\nu_3,\lambda_1\nu_3+\lambda_2\nu_2).
		\end{displaymath}
		\item[(ii)] For $r=0,1$ and $s\geq 0$, the basis elements of $D_2^{r,s}$ are the following:
		\renewcommand{\arraystretch}{1.4}
	\begin{small}	\begin{figure}[H]
			\centering
			\begin{tabular}{r|c|c|}\hline
				$2i+1\geq 2p+1$ & $\lambda_1\lambda_2^i,\quad \lambda_2^{i-1}\nu_3$ & $\lambda_2^{i-p+1}\mu_{2p},\quad \lambda_1\lambda_2^{i-p+1}\mu_{2p-1}$ \\ \hline
				$2i\geq 2p$ & $\lambda_2^i,\quad \lambda_2^{i-1}\nu_2$ & $\lambda_1\lambda_2^{i-p}\mu_{2p},\quad \lambda_2^{i-p+1}\mu_{2p-1}$\\ \hline
				$3\leq 2i+1<2p$ & $\lambda_1\lambda_2^i,\quad \lambda_2^{i-1}\nu_3$ & $\mu_{s+1},\quad \lambda_1\mu_s$ \\ \hline
				$3\leq 2i<2p$ & $\lambda_2^i, \quad\lambda_2^{i-1}\nu_2$ & $\mu_{s+1},\quad \lambda_1\mu_s$ \\ \hline
				$1$ & $\lambda_1$ & $\mu_2$ \\ \hline
				$0$ & 1 & $\gamma_1$ \\ \hline
				$s$ & $D_2^{0,s}$ & $D_2^{1,s}$ 
			\end{tabular}
		\end{figure}\end{small}\noindent For $r\geq 2$ and $s\geq 0$, we have that $D_2^{r,s}=D_2^{r-2,s}\gamma_2$. 

		\item[(iii)] We can write $E_2=K[\nu_{2p}]\otimes D_2$. Furthermore, $E_2$ is generated by the elements  $$\lambda_1,\lambda_2, \nu_2,\nu_3,\nu_{2p}, \gamma_1,\gamma_2,\mu_2,\dotsc,\mu_{2p}.$$ 
	\end{enumerate}
\end{theorem}
\begin{proof}
	The first statement can be obtained as in \cite[Corollary 4 (iii)]{Siegel96} and the remaining assertions follow from Propositions \ref{Prop: Basis D2} and \ref{Prop: Multiplication generators E2}.
\end{proof}

We encapsulate the previous result in the following table:

\renewcommand{\arraystretch}{1.9}
\begin{small}\begin{figure}[H]
	\centering
	\begin{tabular}{r|c|c|c|
		}
		\hline
		$2p$	& $\boxed{\vphan\nu_{2p}} \quad \lambda_2^p \quad \lambda_2^{p-1}\nu_2$ & ${\nu_{2p}\gamma_1}\quad {\lambda_1\mu_{2p}}\quad {\lambda_2\mu_{2p-1}}$ & ${\nu_{2p}\gamma_2}\quad {\lambda_2^p\gamma_2} \quad {\lambda_2^{p-1}\nu_2\gamma_2}$ 
		\\ \hline
		$2p-1$	& ${\lambda_1\lambda_2^{p-1}} \quad {\lambda_2^{p-2}\nu_3}$ & $\boxed{\vphan\mu_{2p}}\quad {\lambda_1\mu_{2p-1}}$ & ${\lambda_1\lambda_2^{p-1}\gamma_2} \quad {\lambda_2^{p-2}\nu_3\gamma_2}$ 
		\\ \hline
		$2p-2$	& ${\lambda_2^{p-1}}\quad {\lambda_2^{p-2}\nu_2}$ & $\boxed{\vphan\mu_{2p-1}}\quad{\lambda_1\mu_{2p-2}}$ & ${\lambda_2^{p-1}\gamma_2}\quad {\lambda_2^{p-2}\nu_2\gamma_2}$ 
		\\ \hline 
		$\vdots$	& $\vdots$ & $\vdots$ & $\vdots$ 
		\\ \hline 
		4	& ${\lambda_2^2}\quad {\lambda_2\nu_2}$ & $\boxed{\vphan\mu_5}\quad{\lambda_1\mu_4}$ & ${\lambda_2^2\gamma_2}\quad {\lambda_2\nu_2\gamma_2}$ 
		\\ \hline 
		3	& $\boxed{\vphan\nu_3} \quad {\lambda_1\lambda_2}$ & $\boxed{\vphan\mu_4}\quad{\lambda_1\mu_3}$ & ${\nu_3\gamma_2}\quad {\lambda_1\lambda_2\gamma_2}$ 
		\\ \hline
		2	& $\boxed{\vphan\lambda_2} \quad \boxed{\vphan\nu_2}$ & $\boxed{\vphan\mu_3} \quad {\lambda_1\mu_2}$ & ${\lambda_2\gamma_2} \quad {\nu_2\gamma_2}$ 
		\\ \hline
		1	& $\boxed{\vphan\lambda_1}$ & $\boxed{\vphan\mu_2}$ & ${\lambda_1\gamma_2}$ 
		\\ \hline
		0	& ${1}$ & $\boxed{\vphan\gamma_1}$ & $\boxed{\vphan\gamma_2}$ 
		\\ \hline
		& 0 & 1 & 2 
	\end{tabular}
	\caption{Basis elements of $E_2^{r,s}$ for $0\leq r\leq 2$ and $0\leq s\leq 2p$, with the multiplicative generators highlighted.}
\end{figure}\end{small}
\renewcommand{\arraystretch}{1}

\begin{remark}
 	In \cite[Corollary 4]{Siegel96}, using analogous notation to ours, Siegel obtains that the multiplicative generators of $E_2(\Heis(p))$ are $\lambda_1,\lambda_2, \nu_2,\nu_3,\nu_{2p}, \gamma_1,\gamma_2,\mu_2,\dotsc,\mu_{2p-3}$. 
\end{remark}



\section{Non-direct second differential computations}\label{sec: survivetoE3} 

In this section, we use restriction, inflation and the norm maps to  determine some of the generators of $E_2$ that survive to the infinity page $E_{\infty}$. We fix the following notation: the inclusion homomorphism $M \hooklongrightarrow G$ induces the restriction map 
$$
\res_{G\rightarrow M}\colon \HH^{\bullet}(G)\longrightarrow \HH^{\bullet}(M)
$$
in cohomology and, by a slight abuse of notation, we also write $\res_{G\to M}$ to denote the composition $\HH^{\bullet}(G)\longrightarrow \HH^{\bullet}(M)\longrightarrow \HH^{\bullet}(M)^Q$. 

\begin{proposition}
	The elements $\lambda_1,\lambda_2,\nu_2,\nu_3,\nu_{2p},\gamma_1,\gamma_2,\mu_2,\mu_3$ survive to $E_{\infty}$.
\end{proposition}

\begin{proof}
It is clear that $\gamma_1,\gamma_2\in E_{\infty}$.  Since the extension \eqref{eq: splitextension} splits, the image of the second differential on $E_2^{\bullet, 0}$ is trivial. Consequently, $\lambda_1,\mu_2\in E_{\infty}$. 


For $\lambda_2=\beta_n(\lambda_1)\in E_2^{0,2}=\HH^2(M)^Q$, consider the map $\pi\colon H^1(G;\Z/p^n\Z)\to H^1(G)$ and let $\tilde\lambda_1\in \HH^1(G; \Z/p^nZ)$ be such that 
$\pi(\tilde\lambda_1)=\lambda_1$. It can be readily checked that $\lambda_1= \res_{G\to M}\circ\; \pi(\tilde\lambda_1)$ and thus, 
$$
\lambda_2=(\beta_n \circ \res_{G \to M} \circ \;\pi) (\tilde\lambda_1)=\res_{G\to M} \circ \; \beta_n (\tilde\lambda_1).
$$
This yields that $\lambda_2 \in \text{Im}(\res_{G\to M}) =E_{\infty}^{0,2}$ .


For $\nu_2$, consider the inflation homomorphism $\inf\colon E_2(\Heis(p))\longrightarrow E_2$. In particular, for  $\tilde \nu_2\in E_2^{0,2}(\Heis(p))$ defined analogously to $\nu_2$ (see \cite[Corollary 4]{Siegel96}, where Siegel uses $y_2$), we have that $\nu_2=\inf(\tilde\nu_2)$. By \cite[Theorem 5]{Siegel96}, $\tilde\nu_2\in E_{\infty}(\Heis(p))$, and since the inflation map commutes with differentials, we conclude that $\nu_2\in E_{\infty}$.

 
For  $\mu_3$, consider the subgroup $H=Q\ltimes(C_{p^n}^p\times C_{p^n})$ of $G$. The action of $\sigma$ on $$\HH^{\bullet}(C_{p^n}^p\times C_{p^n})=\Lambda(w_1,\tilde y_1)\otimes  K[w_2,\tilde y_2]$$ is trivial, and so 
$$
E_2(H)=\HH^{\bullet}(Q)\otimes \HH^{\bullet}(C_{p^n}^p\times C_{p^n})=E_{\infty}(H).
$$
The restriction homomorphism $\res_{G\to H}\colon E_2\longrightarrow E_2(H)$ then sends $\mu_3=\bar y_2$ to $$\res_{G\to H}(\mu_3)=\tilde y_2\gamma_1\not=0.$$ Furthermore, $\dd_2(\mu_3)\in \gen{\mu_2\gamma_2}$ and $\res_{G\to H}(\mu_2\gamma_2)=\tilde y_2\gamma_1\gamma_2\not=0.$ Nevertheless, we have that $\dd_2(y_2\gamma_1)=0$ and, as a consequence, $\dd_2(\mu_3)=0$. Hence, $\mu_3\in E_{\infty}$.

Finally, we will study  the generator $\nu_{2p}$. The subgroup $L=C_{p^n}^p\ltimes M$ of $G$ 
		is normal, and so we have that $L\backslash G/M=G/LM=G/L$. Applying the properties in
		\cite[Theorem 6.1.1]{Evens91} of the Evens norm map $\norm$, we obtain that, for any $\varphi\in \HH^{\bullet}(M)$, 
		\begin{align*}
		\res_{G\to M}\big(\norm_{L\to G}(\varphi)\big) &= \prod_{g\in G/L}\norm_{M\to M}\big(\res_{G\to L}(g\cdot\varphi)\big) = \prod_{g\in G/L}\res_{L\to M}(g\cdot\varphi) \\
		& = \prod_{g\in G/L}g\cdot \res_{L\to M}(\varphi).
		\end{align*}
		Moreover, since the action of $\sigma^p$ on  $\HH^{\bullet}(M)=\Lambda(\tilde x_1,\tilde y_1)\otimes K[\tilde x_2,\tilde y_2]$ is again trivial, we have that 
		\begin{displaymath}
		E_2(L)=\HH^{\bullet}(C_{p^n}^p)\otimes \HH^{\bullet}(M)=E_{\infty}(L),
		\end{displaymath}
		and we can write $y_2=\res_{L\to M}(\tilde y_2)$. 
		Therefore, 
		\begin{displaymath}
		\nu_{2p}=z_{2p}=\prod_{g\in C_p} g \cdot y_2 = \res_{G\to M}\big(\norm_{L\to G}(\tilde y_2)\big)
		\end{displaymath}
		and we deduce that $\nu_{2p}\in E_{\infty}$.

\end{proof}

\section{Generalisation of Siegel's result}\label{sec: SiegelGeneralization}

In this section, we explicitly compute the image of the second differential on the remaining generators of $E_2$. To that aim, we employ a generalization of Siegel's result \cite[Corollary 2]{Siegel96}, which is derived from a theorem by Charlap and Vasquez \cite{Charlap69}. To avoid technicalities in the current section, we collect most of the details and computations of the proof of Theorem \ref{Thm: Charlap Vasquez cyclic} in Appendix \ref{Appendix: Siegel}.


We introduce the necessary notation to state our result. Let $P_{\bullet} \longrightarrow K$ be the minimal projective $KM$-resolution and let $V$ be a $KG$-module with trivial $M$ action. Furthermore, for each $g\in Q$, write $P_{\bullet}^g$ for the $K M$-complex with underlying $K$-complex $P_{\bullet}$ and $M$-action given as follows:  for $h\in M$ and $x\in P_{\bullet}$, we set $h\cdot x=h^{g^{-1}}$. Also, for every $i\in \N$, we write $\Hom_{KM}(P_{\bullet},P_{\bullet})_i$ to denote $\prod_{k=0}^{i}\Hom_{KM}(P_{k},P_{k+i})$.

\begin{theorem}\label{Thm: Charlap Vasquez cyclic}
	Let $\alpha\colon P_{\bullet}\longrightarrow P_{\bullet}^{\sigma^{-1}}$ be a $K M$-chain map commuting with the augmentation, and $\tau\in \Hom_{K M}(P_{\bullet},P_{\bullet})_1$ such that $\partial \tau+\tau\partial =1-\alpha^{p^n}$. Suppose that $\zeta\in E_2^{r,s}$ with $r\geq0,s\geq1$ is represented by $f\in \Hom_{K M}(P_s,V)$. Then, $\dd_2(\zeta)$ is represented by $(-1)^rf\circ \tau$.
\end{theorem}
\begin{proof}
	See Appendix \ref{Appendix: Siegel}.
\end{proof}

\subsection{Chain maps $\alpha$ and $\tau$}

The problem of computing $\dd_2$ is reduced to finding appropriate maps $\alpha$ and $\tau$ satisfying the hypotheses in the previous theorem. We start by defining such maps.

Let $P_{\bullet}'\longrightarrow K$ and $P_{\bullet}''\longrightarrow K$ be the minimal projective resolutions of $K$ as a module over $K\gen{a}$ and $K\gen{b}$, respectively. For each $k\geq 0$, let $e_k'$ and $e_k''$ be the basis elements of $P_k'$ and $P_k''$, respectively. We can then write $P_k'=K\gen{a}e_k'$ and $P_k''=K\gen{b}e_k''$, and so $P_{\bullet}=P_{\bullet}'\otimes P_{\bullet}''\longrightarrow K$ is the minimal projective $KM$-resolution of $K$. If we set 

\begin{equation}\label{eq: eij}
e_j^i=\begin{cases}
e_{i-j}'\otimes e_j'', & \text{if }0\leq j\leq i, \\
0, & \text{otherwise},
\end{cases}
\end{equation}
then, for each $k\geq 0$, the elements $e_0^k,\dotsc,e_k^k$ constitute a basis of $P_k$ as a $KM$-module.  Using the duality $\Cohom^{\bullet}(M)\isom \Cohom_{\bullet}(M)^*$ and the fact that $\Cohom_{\bullet}(M)=P_{\bullet}\otimes_{KM} K$ is a quotient of $P_{\bullet}$ via the canonical map $P_{\bullet}\longrightarrow P_{\bullet}\otimes_{KM} K$, with a slight abuse of notation we can identify the elements of $\Cohom^{\bullet}(M)$ as follows:
\begin{equation}\label{eq: represent_e}
 \text{for } i_1,i_2,j_1,j_2\geq 0,\qquad x_1^{i_1}y_1^{j_1}x_2^{i_2}y_2^{j_2}=(e_{j_1+2j_2}^{i_1+j_1+2i_2+2j_2})^*.
\end{equation}
Consider the elements $\rho,\kappa\in KM$ given by
$$\rho=\sum_{0\leq j\leq i<p^n}a^ib^{j}, \quad \kappa=\sum_{i=0}^{p^n-1}(i+1)a^i,$$
and define the maps $\alpha\in\Hom_{K M}(P_{\bullet},P_{\bullet}^{\sigma^{-1}})_0$ and $\tau\in \Hom_{K M}(P_{\bullet},P_{\bullet})_1$ as the homomorphisms that for $0\leq j\leq i<p^n$ satisfy the following equalities:
\begin{align*}
\alpha(e^{2i}_{2j}) & =\sum_{j\leq k\leq i}\binom{k}{j} (e^{2i}_{2k}-\rho e^{2i}_{2k+1}), & \tau(e^{2i}_{2j})&= -(j+1)\kappa e^{2i+1}_{2j+2}, \\
\alpha(e^{2i}_{2j+1}) & =\sum_{j\leq k\leq i}\binom{k}{j} be^{2i}_{2k+1}, & \tau(e^{2i}_{2j+1})& =-(j+1) e^{2i+1}_{2j+3}, \\
\alpha(e^{2i+1}_{2j}) & =\sum_{j\leq k\leq i}\binom{k}{j} (be^{2i+1}_{2k}+ e^{2i+1}_{2k+1}), & \tau(e^{2i+1}_{2j}) &= -(j+1) e^{2i+2}_{2j+2}, \\
\alpha(e^{2i+1}_{2j+1}) & =\sum_{j\leq k\leq i}\binom{k}{j} e^{2i+1}_{2k+1}, & \tau(e^{2i+1}_{2j+1}) &= -(j+1)\kappa e^{2i+2}_{2j+3}.
\end{align*}

\begin{lemma} \label{Lem: alpha tau}
	The maps $\alpha$ and $\tau$ defined as above satisfy the equalities $\partial \alpha -\alpha\partial=0$ and $\partial \tau -\tau \partial =1-\alpha^{p^n}$.
\end{lemma}
\begin{proof}
	See Appendix \ref{Appendix: Computations}.
\end{proof}

\subsection{Direct second differential computations}\label{subsec: differentials}

Using Theorem \ref{Thm: Charlap Vasquez cyclic} and the maps in Lemma \ref{Lem: alpha tau}, we can now compute the second differential of the remaining generators.




\begin{proposition}\label{Prop: d2 mu}
	The second differential of the elements $\mu_4,\dotsc, \mu_{2p}$ is as follows:
	\begin{enumerate}
		\item[(i)] For $2\leq i\leq p$, we have that $$\dd_2(\mu_{2i})=-(i-1)\lambda_1\mu_{2i-2}\gamma_2
		.$$
		
		\item[(ii)] For $2\leq i\leq p-1$, we have that $$\dd_2(\mu_{2i+1})=-i\lambda_1\mu_{2i-1}\gamma_2
		.$$
	\end{enumerate}
\end{proposition}

\begin{proof}
Consider $\mu_{2i+2}=\overbar{y_1y_2^i }\in E_2^{1,2i+1}$ with $1\leq i\leq p-1$, which, by \eqref{eq: represent_e}, is represented by the map $f\colon P_{2i+1} \longrightarrow K$ with $f=(e^{2i+1}_{2i+1})^*$.  We can easily compute $f\circ \tau$ to obtain that, for $0\leq j\leq k<p^n$, we have that 
		\begin{align*}
		(f\circ \tau) (e^{2k+1}_{2j})  &= 0, \quad (f\circ \tau) (e^{2k+1}_{2j+1})  = 0, \quad (f\circ \tau) (e^{2k}_{2j})  = 0, \\
		(f\circ \tau) (e^{2k}_{2j+1}) & = \begin{cases}
		-i, & \text{if }k=i \text{ and } j=i-1, \\
		0, & \text{otherwise}.
		\end{cases} 
		\end{align*}
		Hence, $-(f\circ \tau)=i (e^{2i}_{2i-1})^*$, which represents $-i\lambda_1\mu_{2i}\gamma_2=i\overbar{x_1 y_1 y_2^{i-1}}$. Consequently, $$\dd_2(\mu_{2i+2})=-i\lambda_1\mu_{2i}\gamma_2.
		$$
		Take now $\mu_{2i+1}=\overbar{y_2^i }\in E_2^{1,2i}$ with $2\leq i\leq p-1$, which is represented by the map $f\colon P_{2i} \longrightarrow K$ with $f=(e^{2i}_{2i})^*$. 
		We compute $f\circ \tau$ to obtain that 
		\begin{align*}
		(f\circ \tau) (e^{2k}_{2j}) & = 0, \quad (f\circ \tau) (e^{2k}_{2j+1})  =0, \quad (f\circ \tau) (e^{2k+1}_{2j+1})  = 0,\\
		(f\circ \tau) (e^{2k+1}_{2j}) & =  \begin{cases}
		-i, & \text{if }k=i-1 \text{ and } j=i-1, \\
		0, & \text{otherwise}.
		\end{cases} 
		\end{align*}
		Hence, $-(f\circ \tau)=i (e^{2i-1}_{2i-2})^*$, which represents $-i\lambda_1\mu_{2i-1}\gamma_2=i\overbar{x_1 y_2^{i-1}}$. Hence, $$\dd_2(\mu_{2i+1})=-i\lambda_1\mu_{2i-1}\gamma_2
		.$$
\end{proof}

The proof of the next result is verbatim to the previous one and we leave it to the reader.

\begin{proposition}\label{prop: d2 nu3} The second differential of the element $\nu_3$ is trivial.
\end{proposition}

\section{Third page of the spectral sequence}\label{sec: DescriptionThirdPage}

Using the results in Sections \ref{sec: survivetoE3} and \ref{subsec: differentials}, we can now determine the structure of the third page $E_3$. First, write $D_3=E_3/\gen{\nu_{2p}}$, and define the elements 
\begin{gather*}
\text{ for } 4\leq i\leq 2p+1, \quad \omega_i=-\lambda_1\mu_{i-1}\in E_2^{1,i-1},\\
\omega_{2p+2}=\lambda_2\mu_{2p}\in E_2^{1,2p+1}, \\
\xi_{2p+1}=\lambda_2\mu_{2p-1}\in E_2^{1,2p}.
\end{gather*}
One can easily verify that these elements have trivial second differential, and so they are in fact elements of $E_3$.

\begin{proposition} \label{Prop: Multiplication generators E3}
	Multiplication by the elements $\nu_{2p},\gamma_2,\lambda_2$ induces vector space homomorphisms as follows:
	\begin{enumerate}
		\item[(i)] Multiplication $\cdot\nu_{2p}\colon E_3^{r,s}\longrightarrow E_3^{r,s+2p}$ is injective for all $r,s\geq 0$. As a consequence, $E_{3}=K[\nu_{2p}]\otimes D_{3}$.
		
		\item[(ii)] Multiplication $\cdot\gamma_2\colon E_3^{r,s}\longrightarrow E_3^{r+2,s}$ is surjective for all $r,s\geq 0$, and an isomorphism for all $r\not=1$, as is $\cdot\gamma_2\colon D_3^{1,s}\longrightarrow D_3^{3,s}$ for $s\geq 2p-1$.

		\item[(iii)] Multiplication $\cdot\lambda_2\colon D_3^{r,s}\longrightarrow D_3^{r,s+2}$ is 
		an isomorphism 
		for all $s\geq 2p$.
	\end{enumerate}
\end{proposition}
\begin{proof}
	The proofs of (i) and (ii) are based on the proof of \cite[Corollary 6]{Siegel96}.
	
	We start with the first statement. For $r,s\geq 0$, let $\varphi\in E_2^{r,s}$ be such that $\dd_2(\varphi)=0$, and suppose that $\varphi\nu_{2p}$ is a trivial element in $E_3$, i.e. there exists  $\psi\in E_2^{r-2,s+2p+1}$ such that   $\varphi\nu_{2p}=\dd_2(\psi)$. Then, since $\gen{\nu_{2p}}\cap \dd_2(E_2\setminus \gen{\nu_{2p}})=0$, there exists  $\upsilon\in E_2^{r-2,s+1}$ such that $\psi=\upsilon\nu_{2p}$. Consequently, $\varphi\nu_{2p}=\dd_2(\upsilon)\nu_{2p}$ and, because $\cdot\nu_{2p}\colon E_2^{r,s}\longrightarrow E_2^{r,s+2p}$ is injective (see Proposition \ref{Prop: Multiplication generators E2}(i)), we have that $\varphi=\dd_2(\upsilon)$, i.e. $\varphi=0$ in $E_3$.

		For the next claim, we first show that multiplication by $\gamma_2$ is surjective. Take  $\varphi\in E_2^{r+2,s}$ with  $r,s\geq 0$ such that $\dd_2(\varphi)=0$.  
		By Proposition \ref{Prop: Multiplication generators E2}(ii), there is some $\psi\in E_2^{r,s}$ such that $\varphi=\psi\gamma_2$ in $E_2$. Then, we have that $\dd_2(\psi)\gamma_2=\dd_2(\varphi)=0$ and, because the product $\cdot \gamma_2\colon E_2^{r+2,s-1}\longrightarrow E_2^{r+4,s-1}$ is injective, we deduce that $\dd_2(\psi)=0$, i.e. $\psi$ survives to $E_3$ and $\varphi=\psi\gamma_2$ in $E_3$. 
		
		We will now study the injectivity of the multiplication by $\gamma_2$. Let $\varphi\in E_2^{r,s}$ with $r\not=1$ or $s\geq 2p-1$ such that $\dd_2(\varphi)=0$. Suppose that there exists $\psi\in E_2^{r,s+1}$ such that $\varphi\gamma_2=\dd_2(\psi)$ and we want to deduce that $\varphi=0$. If $r=0$, or if $\varphi\in D_2^{r,s}$ with $s\geq 2p-1$, then $\dd_2(\psi)=0$, and  by the injectivity of $\cdot\gamma_2\colon E_2^{r,s}\longrightarrow E_2^{r+2,s}$ we obtain that $\varphi=0$. Otherwise, if $r\geq 2$ we have that $\psi=\upsilon\gamma_2$ with $\upsilon\in E_2^{r-2,s+1}$. Hence, $\varphi\gamma_2=\dd_2(\upsilon)\gamma_2$ and, because $\cdot\gamma_2\colon E_2^{r,s}\longrightarrow E_2^{r+2,s}$ is injective, we have that $\varphi=\dd_2(\upsilon)$, i.e. $\varphi=0$ in $E_3$. 
		
		 
		Now, let us show that multiplication by $\lambda_2$ is surjective for $s\geq 2p$. Take  $\varphi\in E_2^{r,s+2}$ with  $s\geq 2p$ such that $\dd_2(\varphi)=0$.  
		By Proposition \ref{Prop: Multiplication generators E2}(iii), there is some $\psi\in E_2^{r,s}$ such that $\varphi=\psi\lambda_2$ in $E_2$. Then, we have that $\dd_2(\psi)\lambda_2=\dd_2(\varphi)=0$ and, because the product $\cdot \lambda_2\colon E_2^{r+2,s-1}\longrightarrow E_2^{r+2,s+1}$ is injective, we deduce that $\dd_2(\psi)=0$, i.e. $\psi$ survives to $E_3$ and $\varphi=\psi\lambda_2$ in $E_3$.

		Finally, we show that multiplication by $\lambda_2$ is injective for $s\geq 2p$.  Let $\varphi\in E_2^{r,s}$ with $g\geq 2p$ such that $\dd_2(\varphi)=0$. Suppose that $\varphi\lambda_2=\dd_2(\psi)$ for some $\psi\in E_2^{r-2,s+3}$. Then, as $\gen{\lambda_2}\cap \dd_2(E_2\setminus \gen{\lambda_2})=0$, there exists $\upsilon\in E_2^{r-2,s+1}$ such that $\psi=\upsilon\lambda_2$. Therefore, $\varphi\lambda_2=\dd_2(\upsilon)\lambda_2$ and, because $\cdot\lambda_2\colon E_2^{r,s}\longrightarrow E_2^{r,s+2}$ is injective, we have that $\varphi=\dd_2(\upsilon)$, i.e. $\varphi=0$ in $E_3$.

\end{proof}

We can now fully determine the structure of $E_3$.

\begin{theorem}\label{Thm: Structure E3}
	The structure of the third page can be described as follows:
	\begin{enumerate}
		\item[(i)] For $r\geq 0$ even, we have that $E_3^{r,\bullet}=E_2^{r,\bullet}$. For $r\geq 5$ odd, we have that $D_3^{r,s}=D_3^{r-2,s}\gamma_2$. For $r=1,3$, the basis elements of $D_3^{r,s}$ are the following:
		\renewcommand{\arraystretch}{1.4}
	\begin{small}	\begin{figure}[H]
			\centering
			\begin{tabular}{r|c|c|}\hline
				$2i+1\geq 2p+1$ & $\lambda_2^{i-p+1}\omega_{2p},\quad  \lambda_2^{i-p}\xi_{2p+2}$ & $\lambda_2^{i-p+1}\omega_{2p}\gamma_2,\quad  \lambda_2^{i-p}\xi_{2p+2}\gamma_2$ \\ \hline
				$2i\geq 2p$ & $\lambda_2^{i-p}\omega_{2p+1},\quad  \lambda_2^{i-p}\xi_{2p+1}$ & $\lambda_2^{i-p}\omega_{2p+1}\gamma_2,\quad \lambda_2^{i-p}\xi_{2p+1}\gamma_2$\\ \hline
				$2p-1$ & $\omega_{2p}$ & $\omega_{2p}\gamma_2$ \\ \hline
				$6\leq s<2p-2$ & $\omega_{s+1}$ & $\emptyset$ \\ \hline
				$5$ & $\omega_6
				$ & $\emptyset$ \\ \hline
				$4$ & $\mu_2\nu_3$ & $\emptyset$ \\ \hline
				$3$ & $\lambda_1\mu_3$ & $\emptyset$ \\ \hline
				$2$ & $\mu_3$ & $\mu_3\gamma_2$ \\ \hline
				$1$ & $\mu_2$ & $\mu_2\gamma_2$ \\ \hline
				$0$ & $\gamma_1$ & $\gamma_1\gamma_2$ \\ \hline
				$s$ & $D_3^{1,s}$ & $D_3^{3,s}$ 
			\end{tabular}
		\end{figure}\end{small}\noindent
		\renewcommand{\arraystretch}{1}Additionally, if $p\geq 5$ we have that $\omega_6=\frac{2}{3}\nu_3\mu_3$, and so $D_3^{1,5}=\gen{\nu_3\mu_3}$.

		\item[(ii)] We can write $E_3=K[\nu_{2p}]\otimes D_3$. Furthermore, for $p=3$, the third page $E_3$ is generated by the elements  $$\lambda_1,\lambda_2, \nu_2,\nu_3,\nu_{6}, \gamma_1,\gamma_2,\mu_2,\mu_3,\omega_6,\omega_7,\omega_{8},\xi_{7},$$ and for $p\geq 5$, by the elements  $$\lambda_1,\lambda_2, \nu_2,\nu_3,\nu_{2p}, \gamma_1,\gamma_2,\mu_2,\mu_3,\omega_7,\dotsc,\omega_{2p+2},\xi_{2p+1}.$$ 
	\end{enumerate}
\end{theorem}
\begin{proof}
	We can deduce from Theorem \ref{Thm: Structure E2} that $E_3^{1,s}=\gen{\omega_{s+1}}$ for $3\leq s\leq 2p-1$. Nevertheless, we can easily compute 
	\begin{displaymath}
	-\omega_4=\lambda_1\mu_3, \quad -\omega_5= \nu_3\mu_2, \quad \frac32\omega_6=\nu_3\mu_3.
	\end{displaymath}
	Everything else follows from Propositions \ref{Prop: d2 mu} and \ref{Prop: Multiplication generators E3}.
\end{proof}

\section{To infinity and beyond}\label{sec: InfinityPage}

Our objective in this section is to show that if $p\geq 5$ the spectral sequence $E$ collapses at $E_3$, i.e. $E_3=E_{\infty}$. In order to achieve our goal, we will define two group automorphisms that will help us show that all the differentials starting with $\dd_3$ are trivial. Let $u\in \mathcal{U}(\Z/p^n\Z)$ be a generator, i.e. $u^{p^{n-1}(p-1)}=1$ but $u^i\not=1$ for any $1\leq i<p^{n-1}(p-1)$. For $0\leq i,j,k\leq p^n -1$, we define the group automorphisms $\Phi\colon G\longrightarrow G$ and $\Psi\colon G\longrightarrow G$ by
\begin{displaymath}
\Phi(\sigma^ka^ib^j)=\sigma^{uk}a^ib^{uj}, \quad \text{and} \quad \Psi(\sigma^ka^ib^j)=\sigma^ka^{ui}b^{uj}.
\end{displaymath}
Because $\Phi(M),\Psi(M)\leq M$,  for every $m\geq 2$, there are induced automorphisms $\Phi^*\colon E_m\longrightarrow E_m$ and $\Psi^*\colon E_m\longrightarrow E_m$.  These automorphisms act on the generators of $D_3$ by multiplying each of them by a power of $u$ as described in the following table:
\renewcommand{\arraystretch}{1.5}
\begin{figure}[H]
	\centering
	\begin{tabular}{r|c|c|c|c|c|c|c|c|c|}
		& $\lambda_{i}$ & $\gamma_{i}$ & $\nu_2$ & $\nu_3$ & $\mu_{2i}$ & $\mu_{2i+1}$ & $\omega_{2i}$ & $\omega_{2i+1}$ & $\xi_{2p+1}$ \\ \hline
		$\Phi$ & $1$ & $u$ & $u$ & $u$ & $u^{i+1}$ & $u^{i+1}$ & $u^{i}$ & $u^{i+1}$ & $u^p$\\ \hline
		$\Psi$ & $u$ & $1$ & $u^2$ & $u^2$ & $u^{i}$ & $u^{i}$ & $u^{i}$ & $u^{i+1}$ & $u^p$\\ \hline
	\end{tabular}
\end{figure}
\renewcommand{\arraystretch}{1}


\begin{proposition} \label{Prop: xi in E infinity}
	For $p\geq 5$, the element $\xi_{2p+1}\in E_3$ survives to $E_{\infty}$.
\end{proposition}
\begin{proof}
	Assume by induction that, for $m\geq 3$, $\xi_{2p+1}\in E_m$, and we will show that $\xi_{2p+1}\in E_{m+1}$. Consider first the case $m=2j+1$ with $j\geq 1$. We have that 
	\begin{equation}\label{eqn: differential xi}
	\dd_{2j+1}(\xi_{2p+1})=t_1\lambda_2^{p-j}\gamma_2^{j+1} + t_2\lambda_2^{p-j-1}\nu_2\gamma_2^{j+1}
	\end{equation}
	with $t_1,t_2\in K$. 
	Applying $\Psi$, we obtain that 
	\begin{displaymath}
	u^p\dd_{2j+1}(\xi_{2p+1})=t_1 u^{p-j} \lambda_2^{p-j}\gamma_2^{j+1} + t_2 u^{p-j+1} \lambda_2^{p-j-1}\nu_2\gamma_2^{j+1}
	\end{displaymath}
	and, equating coefficients with those in (\ref{eqn: differential xi}), we get the conditions
	\begin{equation*}
	\begin{cases}
	t_1(1-u^{j}) & = 0, \\
	t_2(1-u^{j-1}) & = 0.
	\end{cases}
	\end{equation*}
	From these, we deduce that $t_1=0$ for all $j\geq 1$, and $t_2=0$ for all $j>1$. If $j=1$, applying $\Phi$ to (\ref{eqn: differential xi}) we deduce that $t_2(1-u^{p-3})=0$ and $t_2=0$ for $p\geq 5$. Therefore, $\xi_{2p+1}\in E_{2j+1}$ survives to $E_{2j+2}$. 
	
	If $m=2j$ with $j\geq 2$, the only case in which the differential might be non-trivial is $j=p$. We have that $$\dd_{2p}(\xi_{2p+1})=t \mu_2\gamma_2^p$$ with $t\in K$. Applying $\Phi$, we obtain that 
	\begin{displaymath}
	u^p\dd_{2p}(\xi_{2p+1})=t u^{p+2}\mu_2\gamma_2^p,
	\end{displaymath} 
	which implies that $t(1-u^2)=0$, and so $t=0$. Therefore, $\xi_{2p+1}\in E_{2j}$ survives to $E_{2j+1}$. 
\end{proof}

\begin{proposition}\label{Prop: nu3 infinity}
For $p\geq 3$, the element $\nu_3 \in E_3$ survives to $E_{\infty}$.
\end{proposition}

\begin{proof} Observe that, for some $t\in K$, we have $\dd_3(\nu_3)=t\mu_2\gamma_3\in \langle \mu_2\gamma_2\rangle$. Applying $\Phi$ we obtain that $\Phi(\dd_3(\mu_3))=tu^2\mu_2\gamma_2$. Then, $t(u^2-1)\mu_2\gamma_2=0$ implies that $t=0$, as desired.
\end{proof}

\begin{proposition}\label{Prop: omega in E infinity}
	For $p\geq 3$, the elements $\omega_6,\omega_7,\dotsc, \omega_{2p+2}\in E_3$ survive to $E_{\infty}$.
\end{proposition}

\begin{proof}
	The proof for the elements $\omega_7,\dotsc, \omega_{2p+2}\in E_3$ with any $p\geq 3$, and for $\omega_6$ with $p=3$, is analogous to the proof of Proposition \ref{Prop: xi in E infinity}, and can be done following the proof of \cite[Theorem 7]{Siegel96}. For $p\geq 5$, it is clear that $\omega_6=\frac23\nu_3\mu_3$ also survives to $E_{\infty}$.
\end{proof}

Therefore, Propositions \ref{Prop: xi in E infinity}, \ref{Prop: nu3 infinity} and \ref{Prop: omega in E infinity}  prove Theorem \ref{thm: TheoremA}, which we state below.

\begin{theorem}\label{thm: mainthm}
	Let $n\geq 2$ and let $p\geq 5$. Then, the LHSss $E$ associated to $G$ collapses in the third page, i.e. $E_3=E_{\infty}$.
\end{theorem}

\begin{remarks}~
	\begin{enumerate}
		\item[(i)] For $p=3$, following the proof of Proposition \ref{Prop: xi in E infinity}, we are only able to show that $\dd_3(\xi_7)= t\lambda_2 \nu_2 \gamma_2^2$ for some $t\in K$. If $t=0$, then $E_3=E_{\infty}$. Otherwise, the spectral sequence does not converge until at least the fourth page. This stands in contrast with \cite[Theorem 5]{Siegel96}, where it is shown that $E_2(\Heis(3))=E_{\infty}(\Heis(3))$.

		\item[(ii)] For $p\geq 5$, combining our result with \cite[Theorem 7]{Siegel96}, we have that $E_3(\Heis(p^n))=E_{\infty}(\Heis(p^n))$ for all $n\geq 1$.
	\end{enumerate}
\end{remarks}

\section{Poincar\'e series}\label{sec: PoincareSeries}

In this section, we will compute the Poincar\'e series of $\Cohom^{\bullet}(G)$, i.e. the power series 
\begin{displaymath}
P(t)=\sum_{k=0}^{\infty}\big(\dim H^{k}(G)\big)t^k=\sum_{k=0}^{\infty}\sum_{r=0}^{k}(\dim E_{\infty}^{r,k-r})t^k.
\end{displaymath}
Let $D_{\infty}=E_{\infty}/\gen{\nu_{2p}}=D_{3}$, which is the subring of $E_{\infty}$ generated by all the generators except for $\nu_{2p}$. Given that $E_{\infty}=K[\nu_{2p}]\otimes D_{\infty}$, in order to obtain the Poincar\'e series of $E_{\infty}$ we only need to compute the Poincar\'e series of $D_{\infty}$ and multiply it by the Poincar\'e series of $K[\nu_{2p}]$. 

For $k\geq 0$, write 
$$
D_{\infty}^k=\bigoplus_{r+s=k}D_{\infty}^{r,s},\quad \text{so that} \quad \dim D_{\infty}^k=\sum_{r=0}^k \dim D_{\infty}^{r,k-r}.
$$  
Then, the Poincar\'e series of $D_{\infty}$ is given by the power series $P_D(t)=\sum_{k=0}^{\infty}(\dim D_{\infty}^{k})t^k$,
and so we first need to obtain the values $\dim D_{\infty}^k$ for each $k\geq 0$. Note that, for every $r,s\geq 0$, the number $\dim D_{\infty}^{r,s}$ is computed in Theorem \ref{Thm: Structure E3}. Indeed, for $i\geq 0$, we have that 
\begin{align*}
\dim D_{\infty}^{1,s}&=\begin{cases}
1, & 0\leq s\leq 2p-1, \\
2, & s\geq 2p,
\end{cases}\quad \quad\quad
\dim D_{\infty}^{2i,s}=\begin{cases}
1, & s=0,1, \\
2, & s\geq 2,
\end{cases}\\
\dim D_{\infty}^{2i+3,s}&=\begin{cases}
1, & s=0,1,2,2p-1, \\
0, & 3\leq s\leq 2p-2,\\
2, & s\geq 2p.
\end{cases}
\end{align*}

This information can be showcased in the following table: 
\renewcommand{\arraystretch}{1.5}
\begin{small}\begin{figure}[H]
	\centering
       \begin{tabular}{r|c|c|c|c|c|c|c|c|}
		\hline
		$2p+1$	& 2 & 2 & 2 & 2 & 2& 2& 2\\ \hline
		$2p$	& 2 & 2 & 2 & 2 & 2& 2& 2\\ \hline
		$2p-1$	& 2 & 1 & 2 & 1 & 2& 1& 2\\ \hline
		$2p-2$	& 2 & 1 & 2 & 0 & 2& 0& 2\\ \hline 
		$\vdots$	& $\vdots$ & $\vdots$ & $\vdots$ & $\vdots$ & $\vdots$& $\vdots$& $\vdots$\\ \hline
		3	& 2 & 1 & 2 & 0 & 2& 0& 2\\ \hline
		2	& 2 & 1 & 2 & 1 & 2& 1& 2\\ \hline
		1	& 1 & 1 & 1 & 1 & 1& 1& 1\\ \hline
		0	& 1 & 1 & 1 & 1 & 1& 1& 1\\ \hline
		& 0 & 1 & 2 & 3 & 4& 5& 6
	\end{tabular}
	\caption{Dimension of $D_{\infty}^{r,s}$ for $0\leq r\leq 6$ and $0\leq s\leq 2p+1$.}
	\label{Fig: Table dim D infinity}
\end{figure}\end{small}
\renewcommand{\arraystretch}{1}

\begin{lemma}\label{Lem: dim D infinity}
	For $k\geq 0$, we have that 
	\begin{displaymath}
	\dim D_{\infty}^k= \begin{cases}
	k+1, &k=0,1, \\
	k+2, &k=2,3, \\
	k+3, &4\leq k\leq 2p, \\
	2k-2p+3, & k\geq 2p+1.
	\end{cases}
	\end{displaymath}
\end{lemma}
\begin{proof}
	The values $\dim D_{\infty}^k$ for $0\leq k\leq 3$ can be easily computed from the table in Figure \ref{Fig: Table dim D infinity}. Let $4\leq k\leq 2p$ and write $k=2i+\varepsilon$ with $\varepsilon=0,1$. Then, we can compute 
		\begin{align*}
			\sum_{r=2}^{k-3}\dim D_{\infty}^{r,k-r}&=2(i-2+\varepsilon)=k-4+\varepsilon,\\
			\dim D_{\infty}^{k-2,2}&=2-\varepsilon.
		\end{align*}	
		Therefore, we obtain that 
		\begin{align*}
		\dim D_{\infty}^k&=\sum_{r=2}^{k-3}\dim D_{\infty}^{r,k-r}+\dim D_{\infty}^{k-2,2}+5= (k-4+\varepsilon)+(2-\varepsilon)+5= k+3.
		\end{align*}
		Let now $k\geq 2p+1$ and write $k=2i+\varepsilon$ with $\varepsilon=0,1$. Then, we can compute the following values:
		\begin{align*}
		\sum_{r=0}^{k-2p}\dim D_{\infty}^{r,k-r}&=2(k-2p+1)=2k-4p+2, 	& \dim D_{\infty}^{k-2p+1,2p-1}&=1+\varepsilon,\\
		\sum_{r=k-2p+2}^{k-3}\dim D_{\infty}^{r,k-r}&=2(p-2)=2p-4,& \dim D_{\infty}^{k-2,2}&=2-\varepsilon.
		\end{align*}
		Therefore, we obtain that 
		\begin{align*}
		\dim D_{\infty}^k&=\sum_{r=0}^{k-2p}\dim D_{\infty}^{r,k-r}+\dim D_{\infty}^{k-2p+1,2p-1}+\sum_{r=k-2p+2}^{k-3}\dim D_{\infty}^{r,k-r}+\dim D_{\infty}^{k-2,2}+2\\
		&= (2k-4p+2)+(2p-4)+(1+\varepsilon)+(2-\varepsilon)+2 \\
		& = 2k-2p+3.
		\end{align*}
\end{proof}

As a result, we can compute the Poincar\'e series of $\Cohom^{\bullet}(G)$.

\begin{theorem}
	The Poincar\'e series of $\Cohom^{\bullet}(G)$ is 
	\begin{displaymath}
	P(t)=\frac{1+t^2-t^3+t^4-t^5+t^{2p+1}}{(1-t)^2(1-t^{2p})}.
	\end{displaymath}
\end{theorem}
\begin{proof}
	Using Lemma \ref{Lem: dim D infinity}, we can compute the Poincar\'e series for $D_{\infty}$ as follows:
	\begin{align*}
	P_{D}(t)& =\sum_{k=0}^{\infty}(\dim D_{\infty}^k) t^k \\
	& =1+2t+4t^2+5t^3+\sum_{k=4}^{2p}(k+3)t^k+\sum_{k=2p+1}^{\infty}(2k-2p+3)t^k \\
	& = \frac{1+t^2-t^3+t^4-t^5+t^{2p+1}}{(1-t)^2}.
	\end{align*}
	Therefore, because $E_{\infty}=K[\nu_{2p}]\otimes D_{\infty}$, we have that 
	\begin{displaymath}
	P(t)=\frac{P_{D}(t)}{(1-t^{2p})}=\frac{1+t^2-t^3+t^4-t^5+t^{2p+1}}{(1-t)^2(1-t^{2p})}.
	\end{displaymath}
\end{proof}

\section{Conclusion and further questions}\label{sec: conclusion}

We follow the notation introduced in Section \ref{sec: notation}. As a consequence of Theorem \ref{thm: mainthm}, we obtain that, for a prime number $p\geq 5$, the LHSss $E$ of $G$ are isomorphic from the second page on as bigraded $K$-algebras. We have not however determined the ring structure of $\HH^{\bullet}\big(\Heis(p^n)\big)$ and we encourage the ambitious reader to do so. 

Assume now that $K$ is a finite field of characteristic $p$. Then, by \cite[Theorem 2.1]{Carlson05}, there are finitely many liftings of $E_{\infty}\big(\Heis(p^n)\big)$ to the cohomology ring $\HH^{\bullet}\big(\Heis(p^n)\big)$. This in particular yields the following result.

\begin{corollary}\label{cor: conclusion}
Let $p\geq 5$ be a prime number. Then, there are only finitely many isomorphism types of $K$-algebras in the infinite collection $\{\HH^{\bullet}(\Heis(p^n))\}_{n\geq 1}$.
\end{corollary}


The above result is in slight analogy with the previously obtained results in the area \cite{Carlson05}, \cite{DGG17}, \cite{DGG18}, \cite{GG19}, \cite{Symonds21}. Let $\mathbb{G}(-)$ denote an affine group scheme over a ring. For example, the Heisenberg group $\widehat{G}$ and the group $G$ are obtained by applying  such a functor $\mathbb{G}(-)$ to $\Z$ and to $\Z/p^n\Z$, respectively. The presentation of the cohomology rings of such groups is intrinsically hard to obtain. For instance, in \cite{Quillen72}, Quillen described the cohomology rings of the general linear groups $\GL_n(K)$ over a field $K$ of characteristic $p$ with coefficients in a finite field $F$ of characteristic coprime to $p$. However, the case where $K$ and $F$ have the same characteristic is widely open. Based on Corollary \ref{cor: conclusion}, we ask whether the following conjecture holds or not. 

\begin{Conjecture}\label{conjecture}
Let $p$ be a prime number and let $\mathbb{G}(-)$ be an affine group scheme over the $p$-adic integers $\Z_p$. Then, there exists a natural number $f=f(p, \mathbb{G})$ that depends only on $p$ and on $\mathbb{G}$, such that for each $p$ and for all $n\geq f$, the cohomology rings $\HH^{\bullet}(\mathbb{G}(\Z_p/p^{n}\Z_p);K)$ are isomorphic, where $K$ is a field of characteristic $p$ with trivial $\mathbb{G}(\Z_p/p^n\Z_p)$-action.
\end{Conjecture}

The first reason to support the previous conjecture is that the Quillen categories of the groups $\mathbb{G}(\Z_p/p^n\Z_p)$ are isomorphic. That is, the cohomology rings $\HH^{\bullet}(\mathbb{G}(\Z_p/p^n\Z_p);K)$ are $F$-isomorphic (see \cite{Quillen71}). Secondly, observe that for each $n\geq 2$, there is an extension 
\[
\mathbb{G}^{1}(\Z_p/p^n\Z_p) \rightarrow \mathbb{G}(\Z_p/p^n\Z_p)\rightarrow \mathbb{G}(\Z_p/p\Z_p),
\]
where $\mathbb{G}^{1}(\Z_p/p^n\Z_p)$ denotes the first congruence subgroup of $\mathbb{G}(\Z_p/p^n\Z_p)$. It is known that $\mathbb{G}^{1}(\Z_p/p^n\Z_p)$ is a powerful $p$-central group with the $\Omega$-extension property and thus, for every $n\geq 2$, the cohomology rings $\HH^{\bullet}(\mathbb{G}^1(\Z_p/p^n\Z_p);K)$ are isomorphic (\cite{Weigel00}). Moreover, the actions of $\mathbb{G}(\Z_p/p\Z_p)$ on $\HH^{\bullet}(\mathbb{G}^{1}(\Z_p/p^n\Z_p);K)$ are isomorphic, in the sense of \cite[Definition 5.5]{DGG17}.  In turn, the spectral sequences $E_2(\mathbb{G}(\Z_p/p^n\Z_p))$ are isomorphic as bigraded $K$-algebras. Therefore, based on \cite[Conjecture 6.1]{DGG17}, we would expect that the above conjecture holds by taking $f$ to be equal to $2$.


\newpage

\begin{appendix}

\appendix

\section{Generalisation of Siegel's result}\label{Appendix: Siegel}


In this section, we will state a theorem by Charlap and Vasquez \cite{Charlap69} regarding the computation of the second differential of the LHSss associated to a split extension of finite groups and then provide a generalization of \cite[Corollary 2]{Siegel96} for split extensions of cyclic $p$-groups. 

We start by introducing the necessary definitions and notation to state the aforementioned result by Charlap and Vasquez. Let $G=Q\ltimes M$ be a split extension of $Q$ by the finite group $M$ and let $V$ be a $K G$-module with trivial $M$-action.

Let $X_{\bullet}\longrightarrow K$ be a projective $K G$-resolution, let $Y_{\bullet}\longrightarrow K$ be the $K Q$-bar resolution and let $P_{\bullet}\longrightarrow K$ be the minimal $K M$-resolution. If $E=E(G)$ is the LHSss associated to the split extension of $Q$ by $M$, the following identifications hold (\cite[Section 7.2]{Evens91}):
\begin{align}\label{eq: E1 description}
E_0 &= \Hom_{KQ}\big(Y_{\bullet}, \Hom_{KM}(X_{\bullet},V)\big),\nonumber\\
E_1 &= \Hom_{KQ}\big(Y_{\bullet}, \Hom_{KM}(P_{\bullet},V)\big).
\end{align}
For each $g\in Q$, we write $P_{\bullet}^g$ for the $K M$-complex with underlying $K$-complex $P_{\bullet}$ and $M$-action given by  
$$
\text{for} \; \; h\in M \; \text{ and }\;  x\in P_{\bullet}, \; \text{ set }\;\; h\cdot x=h^{g^{-1}}x.
$$
Also, for every $i\in \N$, we write $\Hom_{KM}(P_{\bullet},P_{\bullet}^g)_i$ to denote $\prod_{k=0}^{i}\Hom_{KM}(P_{k},P_{k+i}^g)$. Then, for each $g,g'\in Q$ the Comparison Theorem guarantees (see \cite[Theorem 2.4.2 ]{Benson91} and subsequent remark) the existence of maps $A(g)\in \Hom_{K M}(P_{\bullet},P_{\bullet}^{g})_0$ and $U(g,g')\in \Hom_{K M}(P_{\bullet},P_{\bullet}^{gg'})_1$ satisfying the following conditions:
\begin{enumerate}
	%
	\item[(i)] $\partial A(g)-A(g)\partial=0$ and $\varepsilon A(g)-\varepsilon=0$,
	
	\item[(ii)] $\partial U(g,g')+U(g,g')\partial=A(gg')-A(g)A(g')$. 
\end{enumerate}

\begin{theorem}[{\cite[Theorem 1]{Siegel96}}]\label{thm: SiegelAppendix}
	Let $A$ and $U$ as above. Let $r\geq0,s\geq1$ and suppose that $\zeta\in E_2^{r,s}$ is represented by $f\in \Hom_{K M}(P_s,V)$. Then $\dd_2(\zeta)$ is represented by $(-1)^rD_2(f)$, where 
	\begin{displaymath}
	D_2(f)[g_1|\dotsb |g_{r+2}]=g_1g_2\circ f[g_3|\dotsb |g_{r+2}]\circ U(g_2^{-1},g_1^{-1}).
	\end{displaymath}
\end{theorem}

Although the previous result is for a split extension of a general finite group $Q$, it requires the use of the $K Q$-bar resolution of $K$.  In \cite{Siegel96}, the previous result has been extended for the minimal resolution of a cyclic group $Q$ of size $p$. We  generalise Siegel's result to the case where $Z_{\bullet}\longrightarrow K$ is the minimal $K Q$-resolution  with $Z_k=KQ e_k$, for $k\geq 0$, and where
 $Q=C_{p^n}$ is a cyclic $p$-group of size $p^n$, with $n\geq1$.

\subsection{Proof of Theorem 5.1}\label{sec: GeneralisationSiegel}

The aim of this section is to finish the proof of Theorem \ref{thm: generalisationappendix}. We follow the notation introduced in the beginning of Appendix \ref{Appendix: Siegel} and additionally assume that $Z_{\bullet}\longrightarrow K$ is the minimal $K Q$-resolution  with $Z_k=KQ e_k$, for $k\geq 0$, and where $Q=C_{p^n}$ is a cyclic $p$-group of size $p^n$, with $n\geq1$.  
Under those hypotheses, the first page of the LHSss described in \eqref{eq: E1 description} can be identified with 
\begin{displaymath}
E_1 = \Hom_{KQ}\big(Z_{\bullet}, \Hom_{KM}(P_{\bullet},V)\big).
\end{displaymath}
In order to use Theorem \ref{thm: SiegelAppendix} for the above description of the spectral sequence, we first need explicit chain maps between the bar resolution $Y_{\bullet}$ and the minimal resolution $Z_{\bullet}$. For that purpose, we define the following maps:
	\begin{enumerate}
		\item[(i)] For $k\geq 1$ and $0\leq i_1,\dotsc,i_{2k+1}\leq p^n-1$, let $\theta\colon Y_{\bullet}\longrightarrow Z_{\bullet}$ be a $K$-map that satisfies the next identifications: 
		\begin{align*}
		\theta[]&=e_0, \\
		\theta[\sigma^{i_1}]&=e_1, \\
		\theta[\sigma^{i_1}| \dotsb| \sigma^{i_{2k}}] &=\begin{cases}
		e_{2k}, & \text{if }i_{2j-1}+i_{2j}\geq p^n\text{ for all } 1\leq j\leq k, \\
		0, & \text{otherwise},
		\end{cases} \\
		\theta[\sigma^{i_1}| \dotsb| \sigma^{i_{2k+1}}] &=\begin{dcases}
		\sum_{i=0}^{i_1-1}\sigma^i e_{2k+1}=N_{i_1}(\sigma)e_{2k+1}, & \text{if }i_{2j}+i_{2j+1}\geq p^n\text{ for all } 1\leq j\leq k, \\
		0, & \text{otherwise}.
		\end{dcases}
		\end{align*}

		\item[(ii)] For $k\geq 1$, let $\eta\colon Z_{\bullet} \longrightarrow Y_{\bullet}$ be a $K$-map that satisfies the following identifications:
		\begin{align*}
		\eta(e_0) & = [], \\
		\eta(e_1) & = [\sigma], \\
		\eta(e_{2k}) & = \sum_{0\leq i_1,\dotsc,i_{k}< p^n} [\sigma^{i_1}| \sigma| \dotsb| \sigma^{i_k}| \sigma], \\
		\eta(e_{2k+1}) & = \sum_{0\leq i_1,\dotsc,i_{k}< p^n} [\sigma| \sigma^{i_1}| \dotsb| \sigma | \sigma^{i_k}| \sigma]. \\
		\end{align*}
	\end{enumerate}

\begin{lemma}\label{Lem: Chain maps C-V}
	The above maps $\theta$ and $\eta$ are $K$-chain maps.
\end{lemma}
\begin{proof} We start by showing that $\theta$ is a chain map. To that aim, we need to show that for all $k\geq 1$, the following equalities hold $(\partial\theta-\theta\partial)(Y_{2k})=(\partial\theta-\theta\partial)(Y_{2k+1})=0$. We will only show the equality for the even case, $Y_{2k}$, as the odd case follows similarly. Observe that, for every $1\leq j\leq k-1$ such that $i_{2j}+i_{2j+1}\geq p^n$ and $i_{2j+1}+i_{2j+2}\geq p^n$, we have that 
	\begin{displaymath}
	(i_{2j}+i_{2j+1} \mod p^n)+ i_{2j+2}= i_{2j}+i_{2j+1} -p^n + i_{2j+2} = i_{2j}+ (i_{2j+1} + i_{2j+2}\mod p^n),
	\end{displaymath}
	and thus 
	\begin{equation}\label{eqn: Theta equal}
	\theta[\sigma^{i_1}| \dotsb| \sigma^{i_{2j}+i_{2j+1}}| \sigma^{i_{2j+2}}|\dotsb|\sigma^{i_{2k}}] =\theta[\sigma^{i_1}| \dotsb| \sigma^{i_{2j}}| \sigma^{i_{2j+1}+i_{2j+2}}|\dotsb|\sigma^{i_{2k}}].
	\end{equation}
	Also note that, if there is some $1\leq l\leq k-1$ such that $i_{2l}+i_{2l+1}<p^n$, then
	\begin{equation}\label{eqn: Theta 0}
	\theta[\sigma^{i_1}|\dotsb| \sigma^{i_{2l}}|\dotsb | \sigma^{i_{j}+i_{j+1}} |\dotsb |\sigma^{i_{2k}}]=0,
	\end{equation}
	for every $2l+2\leq j\leq 2k-1$. Therefore, using (\ref{eqn: Theta 0}) we obtain that
	\begin{align}\label{eqn: Theta l}
	\theta\partial[\sigma^{i_1}|\dotsb|\sigma^{i_{2k}}] & = \theta \big(\sigma^{i_1}[\sigma^{i_2}|\dotsb|\sigma^{i_{2k}}]\big) + \sum_{j=1}^{2l+1}(-1)^j\theta[\sigma^{i_1}|\dotsb|\sigma^{i_j+i_{j+1}}|\dotsb|\sigma^{i_{2k}}] \nonumber\\
	& \hspace{30pt} + \sum_{j=2l+2}^{2k-1}(-1)^j\theta[\sigma^{i_1}|\dotsb|\sigma^{i_j+i_{j+1}}|\dotsb|\sigma^{i_{2k}}] +\theta[\sigma^{i_1}|\dotsb|\sigma^{i_{2k}}] \nonumber\\
	& = \theta \big(\sigma^{i_1}[\sigma^{i_2}|\dotsb|\sigma^{i_{2k}}]\big) + \sum_{j=1}^{2l+1}(-1)^j\theta[\sigma^{i_1}|\dotsb|\sigma^{i_j+i_{j+1}}|\dotsb|\sigma^{i_{2k}}].
	\end{align}
	
	Analogously, if there is some $1\leq m\leq k$ such that $i_{2m-1}+i_{2m}<p^n$, then 
	\begin{equation}\label{eqn: Theta m}
	\theta\partial[\sigma^{i_1}|\dotsb|\sigma^{i_{2k}}]= \sum_{j=2m-2}^{2k-1}(-1)^j\theta[\sigma^{i_1}|\dotsb|\sigma^{i_j+i_{j+1}}|\dotsb|\sigma^{i_{2k}}] +\theta[\sigma^{i_1}|\dotsb|\sigma^{i_{2k-1}}].
	\end{equation}
	
	In order to show that equations \eqref{eqn: Theta l} and \eqref{eqn: Theta m} are identical, we need to distinguish four different cases: 
	\begin{enumerate}
		\item[(i)] There is a smallest $l$ with $1\leq l\leq k-1$ such that $i_{2l}+i_{2l+1}<p^n$, and a largest $m$ with $1\leq m\leq k$ such that $i_{2m-1}+i_{2m}<p^n$. 
		\item[(ii)] There is a largest $m$ with $1\leq m\leq k$ such that $i_{2m-1}+i_{2m}<p^n$, but $i_{2j}+i_{2j+1}\geq p^n$ for every $1\leq j\leq k-1$. 
		

		
		\item[(iii)] There is a smallest $l$ with $1\leq l\leq k-1$ such that $i_{2l}+i_{2l+1}<p^n$, but $i_{2j-1}+i_{2j}\geq p^n$ for every $1\leq j\leq k$. 
		
		
		\item[(iv)] For every $1\leq j\leq k$, we have that $i_{2j-1}+i_{2j}\geq p^n$, and for every $1\leq j'\leq k-1$, we have that $i_{2j'}+i_{2j'+1}\geq p^n$ . 
		
		

	\end{enumerate}

We will study the first case carefully and we omit the rest of the cases as the steps to follow are identical. On the one hand, we can easily see that 
		\begin{displaymath}
		\partial\theta[\sigma^{i_1}|\dotsb|\sigma^{i_{2k}}]=0.
		\end{displaymath}
		On the other hand, for $m>l+1$, it is clear that 
		\begin{displaymath}
		\theta\partial[\sigma^{i_1}|\dotsb|\sigma^{i_{2k}}]=0.
		\end{displaymath}
		Furthermore, the equalities in (\ref{eqn: Theta l}) and  (\ref{eqn: Theta m}) yield that, for $2\leq m\leq l+1$,
		\begin{align}\label{eqn: Theta Partial 1}
		\theta\partial[\sigma^{i_1}|\dotsb|\sigma^{i_{2k}}] = \sum_{j=2m-2}^{2l+1}(-1)^j\theta[\sigma^{i_1}|\dotsb|\sigma^{i_j+i_{j+1}}|\dotsb|\sigma^{i_{2k}}].
		\end{align}
		If $2\leq m\leq l$, using (\ref{eqn: Theta equal}), the expression (\ref{eqn: Theta Partial 1}) is reduced to
		\begin{align*}
		\theta\partial[\sigma^{i_1}|\dotsb|\sigma^{i_{2k}}] & = \theta[\sigma^{i_1}|\dotsb|\sigma^{i_{2m-2}+i_{2m-1}}|\dotsb|\sigma^{i_{2k}}]- \theta[\sigma^{i_1}|\dotsb|\sigma^{i_{2m-1}+i_{2m}}|\dotsb|\sigma^{i_{2k}}] \\
		& \hspace{30pt} +\theta[\sigma^{i_1}|\dotsb|\sigma^{i_{2l}+i_{2l+1}}|\dotsb|\sigma^{i_{2k}}]- \theta[\sigma^{i_1}|\dotsb|\sigma^{i_{2l+1}+i_{2l+2}}|\dotsb|\sigma^{i_{2k}}] \\
		& = 0-N_{i_1}(\sigma)e_{2k-1}+ N_{i_1}(\sigma)e_{2k-1} -0 \\
		& = 0.
		\end{align*}
		Likewise, if $m=l+1$ we obtain that 
		\begin{align*}
		\theta\partial[\sigma^{i_1}|\dotsb|\sigma^{i_{2k}}] & =\theta[\sigma^{i_1}|\dotsb|\sigma^{i_{2m-2}+i_{2m-1}}|\dotsb|\sigma^{i_{2k}}]- \theta[\sigma^{i_1}|\dotsb|\sigma^{i_{2m-1}+i_{2m}}|\dotsb|\sigma^{i_{2k}}] \\
		& = N_{i_1}(\sigma)e_{2k-1}- N_{i_1}(\sigma)e_{2k-1}\\
		& =0.
		\end{align*}
		
		Finally, if $m=1$ then
		\begin{align*}
		\theta\partial[\sigma^{i_1}|\dotsb|\sigma^{i_{2k}}] & = \theta \big(\sigma^{i_1}[\sigma^{i_2}|\dotsb|\sigma^{i_{2k}}]\big) + \sum_{j=1}^{2l+1}(-1)^j\theta[\sigma^{i_1}|\dotsb|\sigma^{i_j+i_{j+1}}|\dotsb|\sigma^{i_{2k}}] \\
		& = \theta \big(\sigma^{i_1}[\sigma^{i_2}|\dotsb|\sigma^{i_{2k}}]\big) -\theta[\sigma^{i_1+i_2}|\dotsb|\sigma^{i_{2k}}] \\
		& \hspace{30pt} +\theta[\sigma^{i_1}|\dotsb|\sigma^{i_{2l}+i_{2l+1}}|\dotsb|\sigma^{i_{2k}}]- \theta[\sigma^{i_1}|\dotsb|\sigma^{i_{2l+1}+i_{2l+2}}|\dotsb|\sigma^{i_{2k}}] \\
		& = \sigma^{i_1}N_{i_2}(\sigma)e_{2k-1} - N_{i_1+i_2}(\sigma)e_{2k-1} + N_{i_1}(\sigma)e_{2k-1} -0 \\
		&=0.
		\end{align*}
		
Let us now show that $\eta$ is a chain map. Once again, we will focus on the even case and  only show that  $(\partial\eta-\eta\partial)(e_{2k})=0$ for $k\geq 1$. On the one hand, because the initial sum covers all possible exponents $0\leq i_1,\dotsc,i_{k}< p^n$, it is easy to see that 
	\begin{align*}
	\scalebox{0.85}{$\displaystyle\sum_{0\leq i_1,\dotsc,i_{k}< p^n}\sum_{j=1}^{k-1}[\sigma^{i_1}|\dotsb|\sigma|\sigma^{i_j}|\sigma^{i_{j+1}+1}|\sigma|\dotsb| \sigma^{i_k}| \sigma]$} & = \scalebox{0.85}{$\displaystyle \sum_{{0\leq i_1,\dotsc,i_{k}< p^n}} \sum_{j=1}^{k-1}[\sigma^{i_1}|\dotsb|\sigma|\sigma^{i_j+1}|\sigma^{i_{j+1}}|\sigma|\dotsb| \sigma^{i_k}| \sigma]$}, \\
	\sum_{0\leq i_1,\dotsc,i_{k}< p^n}[\sigma^{i_1}|\dotsb|\sigma|\sigma^{i_k}] & =\sum_{0\leq i_1,\dotsc,i_{k}< p^n}[\sigma^{i_1}|\dotsb|\sigma|\sigma^{i_k+1}],
	\end{align*}
	and so we have that 
	\begin{align*}
	\partial\eta(e_{2k}) &= \partial\bigg(\sum_{0\leq i_1,\dotsc,i_{k}< p^n}[\sigma^{i_1}| \sigma| \dotsb| \sigma^{i_k}| \sigma]\bigg) \\
	& = \scalebox{0.85}{$\displaystyle\sum_{0\leq i_1,\dotsc,i_{k}< p^n}\bigg(\sigma^{i_1}[\sigma|\sigma^{i_2}| \dotsb| \sigma^{i_k}| \sigma] - \sum_{j=1}^{k}[\sigma^{i_1}|\dotsb|\sigma^{i_{j-1}}|\sigma|\sigma^{i_j+1}|\sigma^{i_{j+1}}|\sigma|\dotsb| \sigma^{i_k}| \sigma]$} \\
	& \hspace{50pt} \scalebox{0.85}{$\displaystyle+ \sum_{j=1}^{k-1}[\sigma^{i_1}|\dotsb|\sigma^{i_{j-1}}|\sigma|\sigma^{i_j}|\sigma^{i_{j+1}+1}|\sigma|\dotsb| \sigma^{i_k}| \sigma] + [\sigma^{i_1}|\dotsb|\sigma|\sigma^{i_k}]\bigg)$} \\
	& = \sum_{0\leq i_1,\dotsc,i_{k}< p^n}\sigma^{i_1}[\sigma|\sigma^{i_2}| \dotsb| \sigma^{i_k}| \sigma].
	\end{align*}
	On the other hand,
	\begin{align*}
	\eta\partial(e_{2k}) &= \eta\bigg(\sum_{i=0}^{p^n-1}\sigma^ie_{2k-1}\bigg) \\
	& = \sum_{0\leq i,i_1,\dotsc,i_{k-1}< p^n}\sigma^i[\sigma|\sigma^{i_1}|  \dotsb | \sigma| \sigma^{i_{k-1}}| \sigma].
	\end{align*}
	Therefore, $(\partial\eta-\eta\partial)(e_{2k})=0$.
\end{proof}


We will now state and prove Theorem \ref{Thm: Charlap Vasquez cyclic}.

\begin{theorem}\label{thm: generalisationappendix}
	Let $\alpha\colon P_{\bullet}\longrightarrow P_{\bullet}^{\sigma^{-1}}$ be a $K M$-chain map commuting with the augmentation, and $\tau\in \Hom_{K M}(P_{\bullet},P_{\bullet})_1$ such that $\partial \tau+\tau\partial =1-\alpha^{p^n}$. Suppose that $\zeta\in E_2^{r,s}$ with $r\geq0,s\geq1$ is represented by $f\in \Hom_{K M}(P_s,V)$. Then $\dd_2(\zeta)$ is represented by $(-1)^rf\circ \tau$.
\end{theorem}
\begin{proof}
	The proof of this result can be done by following that of \cite[Corollary 2]{Siegel96}, using the chain maps from Lemma \ref{Lem: Chain maps C-V} and writing $p^n$ instead of $p$ where appropriate. 
\end{proof}

 \subsection{Proof of Lemma \ref{Lem: alpha tau}}\label{Appendix: Computations}

In this section, we will give the explicit computations required in the proof of Lemma \ref{Lem: alpha tau}. To that aim, we display the equalities that will be used during our computations while the proof of such properties is left for the reader.



\begin{lemma} \label{Lem: Identities 1} Let $a,b$ denote the generators of $M$ and let $e_j^i$ be as in \eqref{eq: eij}. 
	\begin{enumerate}
		\item[(i)]  The following identities hold: 
		\begin{align*}
			\rho(b-1) & =bN(ab)-N(a), & \rho(a-1)  &=N(b)-N(ab), \\
			\rho(ab-1) & =N(b)-N(a), & \kappa(a-1)  &=-N(a).
		\end{align*}

		\item[(ii)] The differential of the elements $e_j^i$ is as follows:
		\begin{align*}
			\partial(e^{2i}_{2j}) & =N(a)e^{2i-1}_{2j}+N(b)e^{2i-1}_{2j-1}, & \partial(e^{2i}_{2j+1})  &=(a-1)e^{2i-1}_{2j+1}-(b-1)e^{2i-1}_{2j}, \\
			\partial(e^{2i+1}_{2j}) & =(a-1)e^{2i}_{2j}-N(b)e^{2i}_{2j-1},&  \partial(e^{2i+1}_{2j+1})  &=N(a)e^{2i}_{2j+1}+(b-1)e^{2i}_{2j}.
		\end{align*}
	\end{enumerate}
\end{lemma}

%
%
%
%
%
%
%
%

\vspace{1.5mm}
\begin{proposition}\label{pro: alpha chain map}
	The map $\alpha$ is a chain map, i.e. $\partial\alpha-\alpha\partial=0$.
\end{proposition}
\begin{proof}
	We will only check that $(\partial\alpha-\alpha\partial)(e^{2i}_{2j})=0$ as the other cases follow similarly. We will use Lemma \ref{Lem: Identities 1} during the computations. On the one hand,
	\begin{small}\begin{align*}
		\partial\alpha(e^{2i}_{2j}) & =\sum_{j\leq k\leq i}\binom{k}{j} \partial(e^{2i}_{2k}-\rho e^{2i}_{2k+1}) \\
		& = \sum_{j\leq k\leq i}\binom{k}{j} \Big(\big(N(a)+\rho(b-1)\big)e^{2i-1}_{2k} +N(b)e^{2i-1}_{2k-1} -\rho(a-1)e^{2i-1}_{2k+1}\Big) \\
		& = \sum_{j\leq k\leq i}\binom{k}{j} \Big(bN(ab)e^{2i-1}_{2k} +N(b)e^{2i-1}_{2k-1} +\big(N(ab)-N(b)\big)e^{2i-1}_{2k+1}\Big) \\
		& = \sum_{j\leq k\leq i}\binom{k}{j} \big(bN(ab)e^{2i-1}_{2k} +N(ab)e^{2i-1}_{2k+1}\big) + \sum_{j\leq k+1\leq i}\left[\binom{k+1}{j}-\binom{k}{j}\right] N(b)e^{2i-1}_{2k+1} \\
		& = \sum_{j\leq k\leq i}\binom{k}{j} \big(bN(ab)e^{2i-1}_{2k} +N(ab)e^{2i-1}_{2k+1}\big) + \sum_{j\leq k+1\leq i}\binom{k}{j-1} N(b)e^{2i-1}_{2k+1}.
	\end{align*}\end{small}
	On the other hand,
	\begin{small}\begin{align*}
		\alpha\partial(e^{2i}_{2j}) & = \alpha\big(N(a)e^{2i-1}_{2j}+N(b)e^{2i-1}_{2j-1}\big) \\
		& = \sum_{j\leq k\leq i}\binom{k}{j} N(ab)(be^{2i-1}_{2k}+e^{2i-1}_{2k+1})+\sum_{j\leq k\leq i}\binom{k}{j-1} N(b)e^{2i-1}_{2k+1} \\
		& = \sum_{j\leq k\leq i}\binom{k}{j} N(ab)(be^{2i-1}_{2k}+e^{2i-1}_{2k+1})+\sum_{j\leq k+1\leq i}\binom{k}{j-1} N(b)e^{2i-1}_{2k+1}.
	\end{align*}\end{small}
	Therefore, $(\partial\alpha-\alpha\partial)(e^{2i}_{2j})=0$.
\end{proof}

We are left to prove that the identity  $\partial\tau+\tau\partial=1-\alpha^{p^n}$ holds. In order to do that, we first show the identities that will be used throughout the proof.

\begin{lemma} \label{Lem: Identities 2}~
	\begin{enumerate}
		\item[(i)] We have that $\sum_{r=0}^{p^n-1}\rho^{\sigma^r}b^{r}=\kappa N(b)$.
			
		
		\item[(ii)] For any $i,j\geq 0$ and $m\geq 1$, we have that
		\begin{displaymath}
			\sum_{j\leq k\leq l\leq i}m^{k-j}\binom{l}{k}\binom{k}{j}=\sum_{j\leq l\leq i}(m+1)^{l-j}\binom{l}{j}.
		\end{displaymath}
		
	\end{enumerate}
\end{lemma}

\begin{proposition}\label{prop: alpha and tau appendix}
	The maps $\alpha$ and $\tau$ satisfy the identity $\partial\tau+\tau\partial=1-\alpha^{p^n}$.
\end{proposition}
\begin{proof}
	We will only show that $(\partial\tau+\tau\partial)(e^{2i}_{2j})=(1-\alpha^{p^n})(e^{2i}_{2j})$ since the other cases can be done in a similar way. First, we  compute $(\partial\tau+\tau\partial)(e^{2i}_{2j})$ using Lemma \ref{Lem: Identities 1}. On the one hand, 
	\begin{align*}
		\partial\tau(e^{2i}_{2j}) & =\partial\big(-(j+1)\kappa e^{2i+1}_{2j+2}\big) \\
		& = -(j+1)\kappa (a-1)e^{2i}_{2j+2}+(j+1)\kappa N(b)e^{2i+1}_{2j+1} \\
		& = (j+1)N(a)e^{2i}_{2j+2}+(j+1)\kappa N(b)e^{2i+1}_{2j+1} .
	\end{align*}
	On the other hand,
	\begin{align*}
		\tau\partial(e^{2i}_{2j}) & =\tau\big(N(a)e^{2i-1}_{2j}+N(b)e^{2i-1}_{2j-1}\big) \\
		& = -(j+1)N(a) e^{2i}_{2j+2}-jN(b)\kappa e^{2i}_{2j+1}.
	\end{align*}
	As a consequence, 
	\begin{displaymath}
		(\partial\tau+\tau\partial)(e^{2i}_{2j})=\kappa N(b)e^{2i}_{2j+1}.
	\end{displaymath}
	Now, we  compute $(\partial\tau+\tau\partial)(e^{2i}_{2j})$. Applying $\alpha$ repeatedly to $e^{2i}_{2j}$ and using Lemma \ref{Lem: Identities 2}, we obtain that 
	\begin{displaymath}
		\alpha^m(e^{2i}_{2j}) = \sum_{j\leq k\leq i}m^{k-j}\binom{k}{j} \bigg(e^{2i}_{2k}- \sum_{r=0}^{m-1}\rho^{\sigma^r}b^re^{2i}_{2k+1}\bigg)
	\end{displaymath}
	for any $1\leq m\leq p^n$. 
	Therefore,
	\begin{displaymath}
		\alpha^{p^n}(e^{2i}_{2j}) =\sum_{j\leq k\leq i}p^{n(k-j)}\binom{k}{j}\big(e^{2i}_{2k}-\kappa N(b)e^{2i}_{2k+1}\big)=e^{2i}_{2j} -\kappa N(b)e^{2i}_{2j+1},
	\end{displaymath}
	and thus $(\partial\tau+\tau\partial)(e^{2i}_{2j})=(1-\alpha^{p^n})(e^{2i}_{2j})$.
\end{proof}

\end{appendix}


%
%
%
%
%
%
%

\bibliographystyle{plain}

\begin{thebibliography}{99}

\bibitem{Benson91} D. Benson. {\it Representations and Cohomology I: Basic Representation Theory of Finite Groups and Associative Algebras.} Cambridge University Press, 1991.


\bibitem{Brown82} K. S. Brown. {\it Cohomology of groups.} Graduate Texts in Mathematics, Springer--Verlag, 1982.

\bibitem{CarlsonBook03} J. F. Carlson, L. Townsley, L. Valeri--Elizondo and M. Zhang. {\it Cohomology Rings of Finite Groups.} Algebra and Applications, Springer--Verlag, 2003.

\bibitem{Carlson05} J. F. Carlson. Coclass and Cohomology. {\it J. Pure and Appl. Algebra,} 200, 251--266, 2005.

\bibitem{Charlap66} L. S. Charlap and A.T. Vasquez. The cohomology of group extensions. {\it Trans. Amer. Math. Soc.}, 124, 24--40, 1966.

\bibitem{Charlap69}  L.S. Charlap and A.T. Vasquez. Characteristic classes for modules over groups I. {\it Trans. Amer. Math. Soc.}, 127, 533--549, 1969.

\bibitem{DGG17} A. D\'iaz Ramos, O. Garaialde Oca\~na and J. Gonz\'alez-S\'anchez. Cohomology of uniserial $p$-adic space groups, {\it Trans. Amer. Math. Soc.}, 369, 6725--6750, 2017.

\bibitem{DGG18}  A. D\'iaz Ramos, O. Garaialde Oca\~na and J. Gonz\'alez-S\'anchez. Cohomology of $p$-groups of nilpotency class smaller than $p$, {\it J. Group Theory}, 21, 337--350, 2018.

\bibitem{Evens91} L. Evens. {\it The Cohomology of Groups}. Oxford Mathematical Monographs. The Clarendon Press, Oxford University Press, New York, 1991.

\bibitem{GG19} O. Garaialde Oca\~na and J. Gonz\'alez-S\'anchez. Cohomology of finite $p$-groups of fixed nilpotency class, {\it J. Pure Appl. Algebra}, 223, 4667--4676, 2019.

\bibitem{Gaschutz65} W. Gaschutz. Kohomologische Trivialitaten und aussere Automorphismen von $p$-Gruppen, {\it Math. Z.}, 88, 432--433, 1965.

\bibitem{Gaschutz66} W. Gaschutz. Nichtabelsche $p$-Gruppen besitzen aussere $p$-Automorphismen, {\it J. Algebra}, 4, 1--2, 1996.

\bibitem{Guillot18} P. Guillot. {\it A Gentle Course in Local Class Field Theory: Local Number Fields, Brauer Groups, Galois Cohomology}. Cambridge University Press, 2018.

\bibitem{McCleary01} J. McCleary. {\it A User's Guite to Spectral Sequences}. Cambridge University Press, 2nd edition, 2001.


\bibitem{Quillen71} D. Quillen. The spectrum of an equivariant cohomology ring {I}, {II}, {\it 	Ann. Math.}, 94, 549--572, 573--602, 1971.

\bibitem{Quillen72} D. Quillen. On the Cohomology and {$K$}-theory of the General Linear Groups Over a Finite Field.  {\it Ann. Math.}, 96, 552--586, 1972.

\bibitem{Siegel96} S. F. Siegel. The spectral sequence of a split extension and the cohomology of an extraspecial group of order $p^3$ and exponent $p$, {\it J. Pure Appl. Algebra}, 106, 185--198, 1996.

\bibitem{Symonds21} P. Symonds. Rank, Coclass and Cohomology, {\it Int. Math. Research Notices}, 22, 17399--17412, 2021.

\bibitem{Weigel00}T. Weigel. $p$-Central Groups and Poincar\'e Duality, {\it Trans. Amer. Math. Soc.}, 352, 4143--4154, 2000.

\bibitem{Wells71} C. Wells. Automorphisms of group extensions, {\it Trans. Amer. Math. Soc.}, 155, 189--194, 1971.


\end{thebibliography}


\end{document}